\documentclass[12pt,twoside,english]{article}
\usepackage{babel,amssymb,amsthm}
\usepackage{graphicx}
\usepackage{amsmath}
\usepackage{bm}
\usepackage{float, framed}
\usepackage[utf8]{inputenc}
\usepackage[show]{ed}
\usepackage{wrapfig}
\usepackage{stmaryrd}

\newcommand{\punto}{\,\cdot\,}

\newcommand{\triple}[2]{|\!|\!|#1|\!|\!|_{#2}}

\newcommand{\bs}{\boldsymbol}
\newcommand{\bff}{\mathbf}

\newcommand{\R}{\mathbb R}
\newcommand{\CC}{{\mathcal C}}
\newcommand{\EE}{{\mathcal E}}
\newcommand{\KK}{{\boldsymbol\kappa}}

\newcommand{\eps}[1]{\bs\varepsilon(#1)}
\newcommand{\uu}{{\mathbf u}}
\newcommand{\ww}{{\mathbf w}}
\newcommand{\pp}{{\mathbf p}}
\newcommand{\dd}{{\mathbf d}}

\newcommand{\ddiv}{{\mathrm{div}\,}}
\newcommand{\nn}{{\boldsymbol\nu}}

\newcommand{\HH}{{\mathbf H}}
\newcommand{\LL}{{\mathbf L}}

\newcommand{\matR}{\mathbb R^{d\times d}}
\newcommand{\symm}{{\mathrm{sym}}}

\newtheorem{proposition}{Proposition}[section]
\newtheorem{corollary}[proposition]{Corollary}
\newtheorem{lemma}[proposition]{Lemma}
\newtheorem{theorem}[proposition]{Theorem}

\numberwithin{equation}{section}

\title{A problem in control of elastodynamics with piezoelectric effects\thanks{The work of the first author is partially supported by NSF grants DMS-1521590, DMS-1913004 and DMS-1818772 and Air Force Office of Scientific Research under Award NO: FA9550-19-1-0036. The work of the second and the third authors is partially supported by NSF grant DMS-1818867.}}

\author{Harbir Antil\thanks{Department of Mathematical Sciences, George Mason University. {\tt hantil@gmu.edu}}, 
\and
Thomas S. Brown\thanks{Department of \TSB{ Mathematical Sciences, George Mason University. {\tt tbrown62@gmu.edu}}}, 
\and 
Francisco-Javier Sayas\thanks{Department of Mathematical Sciences, University of Delaware. {\tt fjsayas@udel.edu}}
}

\begin{document}

\title{A problem in control of elastodynamics with piezoelectric effects\thanks{The work of the first and second authors is partially supported by NSF grants DMS-1818772 and DMS-1913004 and Air Force Office of Scientific Research under Award NO: FA9550-19-1-0036. The work of the third author is partially supported by NSF grant DMS-1818867.}}

\author{%
{\sc
Harbir Antil\thanks{Corresponding author. Email: hantil@gmu.edu}
} \\[2pt]
Department of Mathematical Sciences, George Mason University\\[6pt]
{\sc Thomas S. Brown}\thanks{Email: tbrown62@gmu.edu}\\[2pt]
Department of Mathematical Sciences, George Mason University\\[6pt]
{\sc and}\\[6pt]
{\sc Francisco-Javier Sayas}\thanks{Email: fjsayas@udel.edu}\\[2pt]
Department of Mathematical Sciences, University of Delaware
}

\maketitle

\begin{abstract}
{We consider an optimal control problem where the state equations are a coupled hyperbolic-elliptic 
system. This system arises in elastodynamics with piezoelectric effects -- the elastic stress tensor 
is a function of elastic displacement and electric potential. The electric flux acts as the control variable and bound constraints on the control are considered. 
We develop a complete analysis for the state equations and the control problem. The requisite 
regularity on the control, to show the well-posedness of state equations, is enforced using the 
cost functional. We rigorously derive the first 
order necessary and sufficient conditions using adjoint equations and further study their 
well-posedness. 
For spatially discrete (time continuous) problems, we show the convergence of our numerical scheme. Three dimensional numerical experiments  are provided showing convergence properties of a fully discrete method and the practical applicability of our approach.} 
{
Hyperbolic-elliptic system, PDE constraint, control constraints, Piezoelectricity, elastic displacement, electric flux, finite element method, error estimates.
}
\end{abstract}


\section{Introduction}

The goal of this paper is the study of an optimal control problem associated to a physical model of transient wave propagation on a piezoelectric material. We will use the normal component of the electric displacement vector on the boundary to control the motion of the entire solid along time. The state equations consist of an elastic wave equation, where the stress depends on the electric field through a three-index tensor, and an electrostatic equilibrium condition for the electric displacement, which depends on the electric field and the elastic strain. This kind of control problem could be used to design materials that need to take on a desired shape at a certain time, or as studied in \cite{PoKhSh2017}, to reduce the vibrations in the material.

Our work includes: the study of the continuous model and of a generic Finite Element semidiscretization in time; the proof of convergence of the semidiscrete solution to the continuous one; the rigorous derivation of the G\^{a}teaux derivative and the continuous and semidiscrete levels, leading to a mesh-independent optimization algorithm; the detailed description of a fully discrete model; and numerical experiments illustrating convergence and showing performance of the method on a three-dimensional simulation. While the physical setting of the problem under study has been simplified to make it approachable, we emphasize that the state equations modeling the piezoelectric wave propagation mimic the behavior of realistic materials considerably well and the setting contains enough challenges to make it interesting for theoretical and practical study. This is a first installment of a long breadth project that will expand to more complex optimal control setting in our future contributions.

Let $\Omega \subset \mathbb{R}^d$ be a bounded Lipschitz domain with boundary $\Gamma$, partitioned into two non-overlapping relatively open sets $\Gamma_D$ and $\Gamma_N$, and let $T > 0$ be a fixed final time. 
The purpose of this paper is to consider an optimal control problem for solid materials with piezoelectric
effects. The state system is governed by a coupled hyperbolic-elliptic system for elastic displacement 
($\uu$) and electric potential $(\psi)$, respectively. Our goal is to devise a strategy to determine the 
unknown electric flux ($z$: control) to be applied to attain certain desired effects by minimizing a
cost functional $\mathcal{J}(\uu,\psi,z)$ subject to the state equations fulfilled by $(\uu,\psi)$
and control constraints $z\in \mathcal Z_\mathrm{ad}$ where $\mathcal Z_\mathrm{ad}$ is the closed 
and convex set of admissible controls.  For a given desired elastic displacement $\uu_d$ a typical example of $\mathcal{J}$ in control
theory is 
\begin{alignat*}{3}
\mathcal J(\uu, z) := & \frac{1}{2} \int_0^T \|\uu(t) - \uu_d(t)\|_\rho^2 \mathrm d t +
\frac{\alpha}{2} \int_0^T \|\dot{z}(t)\|_{\Gamma}^2 \mathrm d t,
\end{alignat*}
where $\alpha > 0$ denotes the cost of the 
control parameter. Moreover, $\|\cdot\|_{\rho}$ and $\|\cdot\|_\Gamma$ respectively denote a mass density weighted $L^2$ norm
on the bounded domain $\Omega$ and the standard $L^2$-norm on its boundary $\Gamma$. The precise definition of the state equations as well as remaining 
variables and operators will be given in Section~\ref{sec:CCP}.  


The study of piezoelectric materials first arose in the late 19th century after the properties were noticed in certain crystal {s} and the full mathematical setting was first formulated in \cite{Voigt1910}.  We use the standard linearized model (c.f. \cite{KhPeGo2008, DeLaOh2009}), where we have included the grounding condition, $G\psi$ due to the problem dealing only with the electric field $\nabla \psi$ both in the interior and on the boundary.   {This grounding condition will be defined more precisely in the next section.  While we sketch the requirements for the well-posedness of state equation in the context of the control problem, we note that the well-posedness of the PDE has previously been studied in \cite{BrSaSa2017,AkNa2002, KaLaMoKa2006, Kaltenbacher2007, MeNi2005, Cimatti2004,HsSaSa2016, ImJo2012} among others.}

There is a rich amount of existing work on optimal control problems 
governed by elliptic and parabolic problems, we refer to the monograph 
\cite{FTroeltzsch_2010a} and the references therein. On the other hand, 
the work on control of hyperbolic equations, especially numerical analysis, is scarce. 
We refer to the monographs \cite{JLLions_1971a,ILasiecka_RTriggiani_2000a} for the optimal 
control of the wave equation. 
Moreover, we refer to \cite{AKroener_KKunisch_BVexler_2011a} for the convergence of 
semismooth Newton methods for the scalar wave equations. In the context of electromagnetic waves, recent work can be found in \cite{TrYo2012, VeYo2016,Yousept2017}.   
For completeness, we also refer to \cite{CBoehm_MUlbrich_2015a, ALechleiter_JWSchlasche_2017a} 
where algorithmic approaches to solve parameter identification problems with linear elastic 
wave equation are considered, see \cite{AKirsch_ARieder_2016a} for a more general setting. 
While others have worked on control problems involving piezoelectric materials, for example  {\cite{LeNoMeSo2010}},  it has been in the context of shape optimization, or placement of piezoelectric actuators as in \cite{XiSh2016}.   To the best of our knowledge ours is the first work that considers the control of transient elastic waves of such a coupled model and provides complete analysis and numerical analysis for the semi-discrete (discrete in space and continuous in time) problem.


The paper is organized as follows: In Section~\ref{sec:CCP} we begin by introducing the relevant
notation and function spaces. We also describe the state equation and introduce the notion of weak
solutions. This is followed by a description of the control problem. Section~\ref{sec:SCP} is 
devoted to the semidiscrete (continuous in time) control problem. We discuss the well-posedness 
of the state and adjoint equations in Section~\ref{sec:PCM}, their proof is stated in Appendix \ref{sup:sec:Pfs}  {where we will also rely on previous results presented in \cite{BrSaSa2017}}. 
A rigorous derivation of the first order necessary optimality conditions is given next. This is followed by a well-posedness and necessary 
optimality system for the semidiscrete problem. In Section~\ref{sec:conv} we discuss the convergence 
and error estimates for our numerical scheme. We conclude with several illustrative numerical 
examples in Section~\ref{sec:num} which confirm our theoretical findings and further show the
practical relevance of our approach.

\section{The control problem and a semidiscretization} \label{sec:CCP}

We set up our problem in a bounded Lipschitz domain $\Omega \subset \R^d$ with boundary $\Gamma$ with outward pointing normal vector $\nn$.  We partition $\Gamma$ into two non-overlapping relatively open sets $\Gamma_D$ and $\Gamma_N$ with the intention of implementing Dirichlet and Neumann conditions on these parts of the boundary respectively. The material properties of $\Omega$ will be described by three tensors: the elastic stress-strain relation, piezoelectric, and permittivity (or dielectric),
\[
\CC\in L^\infty(\Omega;\R^{(d\times d)\times (d\times d)}),
	\qquad
\EE \in L^\infty(\Omega;\R^{(d\times d)\times d}),
	\qquad
\KK \in L^\infty(\Omega;\R_\symm^{d\times d}),
\]
with the following properties holding almost everywhere in $\Omega$
\begin{alignat*}{6}
& \CC\mathrm A\in \matR_\symm
	\quad & & \forall \mathrm A\in \matR,\\
& (\CC \mathrm A ):\mathrm B=\mathrm A:(\CC\mathrm B) 
	\quad & & \forall \mathrm A,\mathrm B\in \matR, \\
& (\CC\mathrm A):\mathrm A\ge c_0 \mathrm A:\mathrm A
	\quad & & \forall \mathrm A\in \matR_\symm,\\
& \EE \mathbf b\in \matR_\symm
	\quad & & \forall \mathbf b\in \R^d,\\
& (\KK\mathbf b)\cdot\mathbf b \ge k_0 |\mathbf b|^2
	\quad & & \forall \mathbf b\in \R^d.
\end{alignat*}
Here the colon represents the Frobenius inner product of matrices, $c_0$ and $k_0$ are positive constants, and we are using $\mathbb R_\mathrm{sym}^{d\times d}$ to be the space of symmetric $d \times d$ matrices with real components. We will also make use of the transpose of the piezoelectric tensor $\EE^\top$, defined by the relation
\[
(\EE^\top\mathrm A)\cdot\mathbf b=\mathrm A:(\EE\mathbf b)
	\qquad\forall\mathrm A\in \matR, \quad\mathbf b\in \R^d.
\]
The density is a strictly positive function $\rho\in L^\infty(\Omega)$. Using the variables $\uu$ and $\psi$ to denote the elastic displacement and electric potential respectively in $\Omega$ and defining the linear strain (or symmetric gradient) operator by the expression $\eps\uu:=\tfrac12(\nabla\uu+\nabla\uu^\top)$ we are ready to formally define the constitutive relations for the stress and electric displacement as 
\[
\sigma(\uu, \psi):= \CC \eps{\uu} + \EE \nabla \psi,\qquad \qquad 
\mathbf d(\uu, \psi):= \EE^\top \eps{\uu} - \KK \nabla \psi.
\]

\subsection{The state equation}

Using the notation $\HH^1(\Omega):= H^1(\Omega)^d$ and $\HH^{1/2}(\Gamma_D) := H^{1/2}(\Gamma_D)^d$ we introduce the trace operator to the Dirichlet part of the boundary, $ \gamma_D: \HH^1(\Omega) \to \HH^{1/2}(\Gamma_D)$, and we define $\HH_D^1(\Omega):= \mathrm{ker} \, \gamma_D$.  {Note that this is a closed subspace of a Hilbert space, and so is itself a Hilbert space (see, for example, \cite[Theorem 3.2-4]{Kreyszig1989})}. This operator can also be thought of as the restriction of the regular trace operator $\gamma:\HH^1(\Omega) \to \HH^{1/2}(\Gamma)$ to $\Gamma_D$.  {In the case that $\Gamma_D = \emptyset$, and so $\Gamma = \Gamma_N$, we take $H_D^1(\Omega) = H^1(\Omega)$.} The normal component of an element $ {\mathbf p \in }\HH(\ddiv,\Omega)$ will be denoted $\mathbf p \cdot \nn$ \cite{GrRa1986}, and the notation $\langle \cdot, \cdot \rangle_\Gamma$ will represent the $H^{-1/2}(\Gamma) \times H^{1/2}(\Gamma)$ duality pairing.   {The notation $\ddiv$ will always be used to mean that we are taking the divergence of a matrix-valued quantity along the rows.}   Additionally, we define $\mathrm{H}_\symm(\ddiv,\Omega):= \{ \mathrm S \in L^2(\Omega;\matR_\symm): \ddiv \mathrm S \in \LL^2(\Omega)\}$ and $\widetilde{\HH}^{1/2}(\Gamma_N):=\{ \gamma \uu : \uu \in \HH_D^1(\Omega)\}$  {(this is the space of traces of $\mathbf H^1(\Omega)$ functions with zero Dirichlet trace)} so that the normal trace 
\[
\gamma_N: \mathrm H_\symm (\ddiv, \Omega) \longrightarrow \HH^{-1/2}(\Gamma_N) := \widetilde{\HH}^{1/2}(\Gamma_N)^\ast,
\]
is the restriction $\gamma_N \mathrm S = \mathrm S\nn |_{\Gamma_N}$, where we are using the asterisk to denote the dual space.  With the notation $\langle \cdot, \cdot \rangle_N$ to represent the $\HH^{-1/2}(\Gamma_N) \times \widetilde{\HH}^{1/2}(\Gamma_N)$ duality pairing, we define $\gamma_N$ with the integration by parts formula  {(in this context referred to as Betti's formula \cite[Section 7.7]{SaBrHa2019})}
\[
\langle \gamma_N \mathrm S, \gamma \mathbf v \rangle_N := (\mathrm S, \eps{\mathbf v})_\Omega + (\ddiv \mathrm S, \mathbf v)_\Omega \qquad \forall \mathbf v \in \HH_D^1(\Omega),
\]
where $(\punto, \punto)_\Omega$ denotes the $L^2$-inner product for matrix-valued, vector-valued, or scalar-valued functions where appropriate.  The space $L_0^2(\Gamma) := \{ z \in L^2(\Gamma) : \int_\Gamma z = 0\}$ will be used throughout as the space in which our control variable (data for the state equation) takes values.  In order to guarantee the uniqueness of the electric potential that solves the state equation (to be defined shortly) we need to introduce the grounding condition operator $G: H^1(\Omega) \to \R$ such that $G$ is linear, bounded, and $G1 \not = 0$.  One possibility ---the one we use in practice--- is $G\psi = \int_\Omega \psi$. 

For data $z:[0,T] \to L_0^2(\Gamma)$ and for every $t \in [0,T]$ the state equations are  
\begin{alignat*}{8}
\rho \ddot{\uu}(t) &= \ddiv (\sigma(\uu(t), \psi(t))), &\qquad  
	\nabla \cdot \mathbf d(\uu(t),\psi(t))&=0 ,\\
\gamma_D \uu(t) & = \mathbf 0, &
	\gamma_N \sigma(\uu(t), \psi(t)) & = \mathbf 0,\\
G\psi(t) & = 0, &
	\mathbf d(t) \cdot \nn &= z(t),\\
\uu(0) &= \mathbf 0, &\quad \dot{\uu}(0) &= \mathbf 0  {,}
\end{alignat*}
 {where equality is to be understood in the distributional sense in the appropriate spaces}.   Although we take homogeneous source and boundary terms (except for $z$), we are easily able to handle the case with non-homogenous terms  {(for more details see Section 3.3 of \cite{Brown2018})}.  For what follows, we will deal with a slightly weaker concept of solution.  To precisely present this idea, we need to introduce the weighted space $\LL_\rho^2(\Omega)$, that is $\LL^2(\Omega)$ using the inner product $(\punto, \punto)_\rho := (\rho \punto, \punto)_\Omega$. 
The space $\HH_D^{-1}(\Omega)$ is the dual of $\HH_D^1(\Omega)$ when we identify $\LL_\rho^2(\Omega)$ with its dual and therefore 
\[
\HH_D^1(\Omega) \subset \LL_\rho^2(\Omega) \subset \HH_D^{-1}(\Omega)
\]
is a well-defined Gelfand triple.  Furthermore we use $\langle \cdot, \cdot \rangle_\rho$ to denote the $\HH_D^{-1}(\Omega) \times \HH_D^1(\Omega)$ duality pairing.  We include the grounding condition in the space  $H_G^1(\Omega): = \{\psi \in H^1(\Omega) : G \psi = 0\}$,  {and notice that in this space we have the norm equivalence $\|\nabla \psi\|_\Omega \approx \|\psi\|_{1, \Omega}$ as a consequence of the Deny-Lions Theorem (see, for example, \cite[Section 7.3]{SaBrHa2019}).}
 {In order to shorten the statement of the problem, we introduce the bilinear form
\begin{alignat*}{3}
a((\uu, \psi), (\ww, \varphi)) := & \; (\CC\eps{\uu} + \EE \nabla \psi, \eps{\ww})_\Omega + (-\EE^\top\eps{\uu} + \KK \nabla \psi, \nabla \varphi)_\Omega,\\
 = & \; (\sigma(\uu,\psi),\eps{\ww})_\Omega - (\mathbf d(\uu, \psi), \nabla \varphi)_\Omega. 
\end{alignat*}}
When we refer to a solution of the state equations, we mean a pair of functions
\begin{subequations} \label{eq:state}
\begin{alignat}{3} 
\uu &\in \CC^0([0,T];\HH_D^1(\Omega)) \cap \CC^1([0,T];\LL_\rho^2(\Omega)) \cap \CC^2([0,T];\HH_D^{-1}(\Omega)),\\
\psi &\in \CC^0([0,T];H_G^1(\Omega)), 
\end{alignat}
such that for all $t \in [0,T]$
\begin{alignat}{4} 
\label{eq:stateA}
\langle \ddot{\uu}(t), \ww \rangle_\rho + a((\uu(t), \psi(t)),(\ww,\varphi)) &= - \langle z(t), \gamma \varphi \rangle_\Gamma  \quad  \forall (\ww, \varphi) \in \HH_D^1(\Omega) \times H_G^1(\Omega),\\
\uu(0)& = \mathbf 0, \quad \dot{\uu}(0) = \mathbf 0 \label{eq:stateB}.
\end{alignat}
\end{subequations}
We remark here that we are using $\HH_D^1(\Omega) \times H_G^1(\Omega)$ as a test space, but this is equivalent to using $\HH_D^1(\Omega) \times H^1(\Omega)$, since $z(t)\in L_0^2(\Gamma)$ for all $t$.  In other words, it does not matter if we test with functions from $H_G^1(\Omega)$ or from the entire space $H^1(\Omega)$, and we will use the two interchangeably. 

\subsection{The control problem}  

Since we will be using the Neumann boundary condition on the electric displacement as control, we need to define $\mathcal Z := \{z \in H^1(0,T; L_0^2(\Gamma)): z(0) = 0\}$ with the norm
\[
\triple{z}{\mathcal Z}^2:= \int_0^T \|\dot{z}(\tau)\|_\Gamma^2 \mathrm d \tau
\] 
making it a Hilbert space, and the admissible set $\mathcal Z_\mathrm{ad} := \{ z \in \mathcal Z: z_a \leq z(t) \leq z_b  \quad a.e. \quad \forall t\},$ where $ z_a \leq 0 \leq z_b$ are constants.  Note that this sign restriction is needed to ensure that $\mathcal Z_\mathrm{ad} \not = \emptyset$. We will use the space $\mathcal U:= \CC^0([0,T];\HH_D^1(\Omega))$ endowed with the norm 
\[
\triple{\uu}{\mathcal U}^2:= \int_0^T \|\uu(\tau)\|_{1,\Omega}^2 \mathrm d \tau,
\] 
as the space for our elastic displacement,  {noting that} this space is not complete  {with respect to this norm}.  We will also make use of the weaker norm \[
\triple{\uu}{\rho}^2: = \int_0^T \|\uu(\tau)\|_\rho^2 \mathrm d \tau
\]
in $\mathcal U$.
As a general rule, and to help the reader handle different norms, triple bars will always be used for norms affecting the space and time variables, while double bars will be used for norms in the space variables (including dual norms). The solution operator for the state equation \eqref{eq:state} is $S: \mathcal Z \longrightarrow \mathcal U$ given by $Sz=\mathbf u$, where the pair $(\mathbf u,\psi)$ satisfies \eqref{eq:state}.

We delay the statement and proof that this operator is well-defined to Section \ref{sec:PCM} and Appendix \ref{sup:sec:Pfs} respectively. 
The desired state for the elastic displacement is a function $\uu_d \in \mathcal U$ such that $\uu_d(0) = 0$. The initial value for the {\em given} desired state is set to zero, matching the one for the state equation. If a desired state were to start from a non-zero value at $t=0$, we would make the state equation start with the same one. The functional we wish to minimize is 
\begin{alignat}{3}
\mathcal J(\uu, z) := & \frac12 \int_0^T \|\uu(t) - \uu_d(t)\|_\rho^2 \mathrm d t + \frac{\alpha}{2} \int_0^T \|\dot{z}(t)\|_{\Gamma}^2 \mathrm d t \label{eq:Jdef}\\
 = & \frac12 \triple{\uu - \uu_d}{\rho}^2 + \frac{\alpha}{2} \triple{z}{\mathcal Z}^2, \nonumber
\end{alignat}
subject to 
\[
Sz = \uu, \qquad \qquad z \in \mathcal Z_\mathrm{ad}.
\]
Here $\alpha$ is a positive constant.  We can rewrite the functional in reduced form by eliminating the restriction given by the state equation:
\begin{equation} \label{eq:func}
j(z) := \mathcal J(Sz,z)
= \frac12 \triple{Sz - \uu_d}{\rho}^2 + \frac{\alpha}{2} \triple{z}{\mathcal Z}^2.
\end{equation}
The control problem can now be stated as
\begin{equation} \label{eq:cont2}
j(z) = \mathrm{min}! \qquad z \in \mathcal{Z}_\mathrm{ad}.
\end{equation}

\subsection{Semidiscretization in space} \label{sec:SCP}

We now shift our perspective to a version of the control problem which has been discretized in space, while kept continuous in time.  The goal of this semidiscretization is to state the problem in such a way that it would be natural to solve the state equation using a Finite Element method. We keep the same geometric setting, but now introduce finite-dimensional subspaces $\bff V_h \subset \HH_D^1(\Omega)$ and $W_h \subset H^1(\Omega)$ with the additional requirement that $W_h$ contain the space of constant functions, i.e., $\mathcal P_0(\Omega) \subset W_h$.  We also define the test space $W_h^G := W_h \cap H_G^1(\Omega)$.  Typically we will have a simplicial mesh of $\Omega$, denoted $\mathcal T_h$, and we will define 
\[
W_h := \{\varphi \in \mathcal C^0( {\overline{\Omega}}): \varphi |_K \in \mathcal P_k(K) \quad \forall K \in \mathcal T_h \}, \qquad \mathbf V_h := \{ \mathbf w \in W_h^d : \mathbf w|_{\Gamma_D} = 0\},
\]
where, for positive integer $k$, $\mathcal P_k$ is the space of polynomials of degree less than or equal to $k$. We emphasize that we will not need any particular choice of $W_h$ and $\mathbf V_h$ for our method to be meaningful, but that we will require some kind of approximation property later on. 

Now given $z \in \CC^0([0,T]; L_0^2(\Gamma))$ such that $z(0)=0$ and $\dot{z} \in L^1(0,T;L_0^2(\Omega))$, we look for 
\begin{subequations} \label{eq:SDstate}
\begin{equation}
(\uu_h, \psi_h) \in \CC^2([0,T];\mathbf V_h) \times \CC^0([0,T];W_h^G)
\end{equation}
that for all $t \in [0,T]$ satisfy
\begin{alignat}{4}
 ( \rho \ddot{\uu}_h(t), \ww )_\Omega + a((\uu_h(t), \psi_h(t)),(\ww,\varphi)) 
&= - \langle z(t), \gamma \varphi \rangle_\Gamma \hspace{5pt} \forall (\ww, \varphi) \in \bff V_h \times W_h^G, \\
 \uu_h(0) = \bff 0, \qquad \dot{\uu}_h(0) &= \bff 0. 
\end{alignat}
\end{subequations}
With the definition of the space $\mathcal U_h:= \CC^0([0,T];\bff V_h)$, the semidiscrete state equation solver $S_h: \mathcal Z \to \mathcal U_h$ is given by $S_hz=\uu_h$, where $(\uu_h,\psi_h)$ solves \eqref{eq:SDstate}. 

The semidiscrete reduced functional $j_h: \mathcal Z \to [0,\infty)$ is given by 
\[
j_h(z) := \frac12 \triple{S_h z - \uu_d }{\rho}^2 + \frac{\alpha}{2} \triple{z}{\mathcal Z}^2=\mathcal J(S_hz,z).
\]
We will also need a semidiscrete control variable.  To define this properly, we create a partition of $\Gamma$, denoted $\Gamma_h$. We take the semidiscrete control to be in the space $\mathcal Z_h:= \{z \in \mathcal Z: z(t) \in \mathcal P_0(\Gamma_h) \quad \forall t\}$, where $\mathcal P_0(\Gamma_h)$ is the space of piecewise constant functions on $\Gamma_h$. In the case where $\Omega$ is a polyhedral domain and we have used Finite Element spaces on a triangulation of $\Omega$ as choices for $W_h$ and $\mathbf V_h$, it is natural (and practical from the point of view of implementation) to set $\Gamma_h$ to be the inherited partition of $\Gamma$, although this is not necessary for the theoretical arguments that follow. We note that $\mathcal Z_h$ is a closed subspace of $\mathcal Z$.  
The admissible set for the semidiscrete control problem is  $\mathcal Z_\mathrm{ad}^h := \mathcal Z_h \cap \mathcal Z_\mathrm{ad}$, so that the control problem is 
\[
j_h(z_h) = \min! \qquad z_h \in \mathcal Z_\mathrm{ad}^h. 
\]

\section{Solvability and optimality conditions} \label{sec:PCM}

It is the goal of this section to provide more details about  the continuous control problem introduced in Section \ref{sec:CCP}.  Whenever we use the symbol $\lesssim$, we will be hiding constants that are independent of the time variable.  Additionally, when we use this symbol in the semidiscrete problem, the constants that we are hiding will be independent of $h$, that is, independent of the choice of the finite-dimensional subspaces.  We now state a theorem about the well-posedness of the state equation \eqref{eq:state}, but save a proof for Appendix \ref{sup:sec:Pfs}.

\begin{theorem} \label{thm:1}
If $z \in \mathcal Z$, then the state equation \eqref{eq:state} has a unique solution that satisfies the bound 
\[
\|\uu(t)\|_{1,\Omega} + \|\psi(t)\|_{1,\Omega} \lesssim \int_0^t \|z(\tau)\|_\Gamma \mathrm d \tau + \int_0^t \|\dot{z}(\tau) \|_\Gamma \mathrm d \tau.
\]
Therefore $S:\mathcal Z\to \mathcal U$ is bounded.
\end{theorem}

We now turn our attention to showing that the control problem is uniquely solvable. 

\begin{theorem} For the continuous control problem discussed in Section \ref{sec:CCP}, the following hold:  \label{prop:wlsc} 
\begin{enumerate}
\item[(a)] the operator $S$ is linear and bounded, 
\item[(b)] the admissible set $\mathcal Z_\mathrm{ad}$ is closed and convex in $\mathcal Z$, hence it is also weakly closed, 
\item[(c)] the functional $j: \mathcal Z \to \R$ defined by \eqref{eq:func} is continuous and (strictly) convex, therefore it is also weakly lower semicontinuous,
\item[(d)] the functional $j: \mathcal Z \to \R$ is coercive. 
\end{enumerate}
Therefore the control problem \eqref{eq:cont2} has a unique weak solution.
\end{theorem}

\begin{proof} Properties (a)-(d) are straightforward to prove. Unique solvability of the control problem follows from the well-known theory of convex optimization on normed spaces (see \cite[Section 7.4]{Ciarlet1982}). 
\end{proof}

\paragraph{Remark} Note that existence of optimal control can also be proved for more general functionals of the form
\[
j(z):=J_1(Sz)+J_2(z),
\]
where $J_1:\mathcal U \to [0,\infty)$ is weakly lower semicontinuous (or, even more generally, if $J_1\circ S:\mathcal Z\to [0,\infty)$ is weakly lower semicontinuous) and $J_2:\mathcal Z\to \mathbb R\cup \{\infty\}$ is proper convex, lower semicontinuous and admitting a lower bound of the form
\[
J_2(z)\ge K_1 \triple{z}{\mathcal Z}+K_2 \quad\forall z\in \mathcal Z,
\]
where $K_1>0$ and $K_2\in \mathbb R$.

\subsection{Adjoint problem and G\^ateaux derivative}

For data $\mathbf f: [0,T] \longrightarrow \HH_D^1(\Omega)$, we look for
\begin{subequations}\label{eq:adj}
\begin{alignat}{3}
\pp &\in \CC^2([0,T];\LL_\rho^2(\Omega)) \cap \CC^1([0,T];\HH_D^1(\Omega)),\\
\xi & \in \CC^1([0,T];H^1(\Omega)),
\end{alignat}
satisfying
\begin{alignat}{5}
\label{eq:adjC}
\rho \ddot{\pp}(t) &= \ddiv(\sigma(\pp(t), \xi(t))) + \rho \mathbf f(t) &\qquad & t \in [0,T],\\
\label{eq:adjD}
0 &= \nabla \cdot \dd(\pp(t), \xi(t)) && t \in [0,T], \\
\label{eq:adjE}
\gamma_D \pp(t)&= \mathbf 0 && t \in [0,T], \\
\label{eq:adjF}
\gamma_N \sigma(\pp(t), \xi(t)) &= \mathbf 0 && t \in [0,T], \\
\label{eq:adjGG}
G\xi(t) &= 0 && t \in [0,T], \\
\label{eq:adjH}
\dd(\pp(t), \xi(t)) \cdot \nn & = 0 && t \in [0,T], \\
\pp(T) &= \mathbf 0, \quad \dot{\pp}(T) = \mathbf 0 \label{eq:adjG}.
\end{alignat}
\end{subequations}
We will refer to \eqref{eq:adj} as the adjoint equations. We will also consider the space $\mathcal X:=L^2(0,T;L^2(\Gamma))$ with norm
\[
\triple{y}{\mathcal X}^2:=\int_0^T \| y(t)\|_\Gamma^2\mathrm dt,
\]
and the operator $R:\mathcal U \to \mathcal X$ given by $R\mathbf f=\gamma\xi$, where $(\pp,\xi)$ solve \eqref{eq:adj}. 

\begin{theorem} \label{thm:2}
For $\mathbf f \in \mathcal U$, \eqref{eq:adj} is uniquely solvable and we have the bound
\[
\|\pp(t)\|_{1,\Omega} + \|\xi(t)\|_{1,\Omega}  \lesssim \int_t^T \|\mathbf f(\tau)\|_\Omega \mathrm d \tau.
\]
Therefore $R:\mathcal U \to \mathcal X$ is bounded. 
\end{theorem}

\begin{proof}
As with Theorem \ref{thm:1}, the proof of Theorem \ref{thm:2} can be found in Appendix \ref{sup:sec:Pfs}.
\end{proof}
  
We note that $j:\mathcal Z\to \mathbb R$ is a continuous quadratic functional, and therefore it is Fr\'echet and G\^{a}teaux differentiable.  
Now {, for $z, y \in \mathcal Z$,} we investigate the G\^{a}teaux derivative 
\begin{equation} \label{eq:grad}
\langle j'(z),y\rangle := \int_0^T ((Sz-\uu_d)(t), Sy(t))_\rho \mathrm d t + \alpha \int_0^T \langle \dot{z}(t), \dot{y}(t)\rangle_\Gamma \mathrm d t.
\end{equation}

\begin{proposition} \label{prop:4.6}
The G\^{a}teaux derivative of $j$ at $z$ in the direction $y\in \mathcal Z$ is
\[
\langle j'(z),y\rangle = \int_0^T \langle R(Sz-\mathbf u_d)(t),y(t) \rangle_\Gamma \mathrm d t + \alpha \int_0^T \langle \dot{z}(t), \dot{y}(t)\rangle_\Gamma \mathrm d t.
\]
\end{proposition}

\begin{proof}
Let  $(\pp, \xi)$ be the solution to the adjoint equation \eqref{eq:adj} with data $\bff f:= \uu - \uu_d =Sz-\uu_d\in \CC^0([0,T];\HH_D^1(\Omega))$ so that $\gamma\xi=R(Sz-\uu_d)$. Let also $(\ww,\eta)$ be the solution to \eqref{eq:state} with data $y$, so that $Sy=\ww$. If we prove that
\begin{equation}
\langle \ddot{\ww}(t), \pp(t) \rangle_\rho - (\ww(t), \rho \ddot{\pp}(t))_\Omega + (\ww(t), \bff f(t))_\rho = \langle y(t), \gamma \xi(t)\rangle_\Gamma, \label{lem:4.1}
\end{equation}
it follows that
\begin{alignat*}{4}
\int_0^T \big((\uu - \uu_d)(t), \ww(t))_\rho \mathrm d t &= \int_0^T \Big((\ww(t), \rho \ddot{\pp}(t))_\Omega - \langle \ddot{\ww}(t), \pp(t)\rangle_\rho \mathrm + \langle y(t), \gamma \xi(t) \rangle_\Gamma\Big) \mathrm d t\\
& = \int_0^T \langle y(t), \gamma \xi(t) \rangle_\Gamma \mathrm d t,
\end{alignat*}
where we are able to eliminate the terms with two time derivatives by integrating by parts and using \eqref{eq:stateB} and \eqref{eq:adjG}. This reconciles the direct expression for the G\^ateaux derivative \eqref{eq:grad} with the formula given in the statement of the Proposition.

To show \eqref{lem:4.1}, we begin by  using integration by parts to see that for all $(\mathbf v, \varphi) \in \HH_D^1(\Omega) \times H^1(\Omega)$, we have 
\[
\langle \ddot{\ww}(t), \mathbf v \rangle_\rho + (\eps{\ww(t)}, \sigma(\mathbf v, \varphi))_\Omega + (\nabla \eta(t), \mathbf d(\mathbf v, \varphi))_\Omega = \langle y(t), \gamma \varphi \rangle_\Gamma.
\]
Testing with solution $(\pp(t) , \xi(t))$, we have 
\begin{equation} \label{eq:4.2}
\langle \ddot{\ww}(t), \pp(t)\rangle_\rho + (\eps{\ww(t)}, \sigma(\pp(t), \xi(t)))_\Omega + (\nabla \eta(t), \dd(\pp(t), \xi(t)))_\Omega = \langle y(t), \gamma \xi(t) \rangle_\Gamma.
\end{equation}
Noting that 
\[
(\nabla \eta(t), \mathbf d(\pp(t), \xi(t))))_\Omega = - (\eta(t),\nabla \cdot \mathbf d(\pp(t), \xi(t))_\Omega + \langle \mathbf d(\pp(t), \xi(t))\cdot \nn, \eta(t) \rangle_\Gamma = 0,
\]
we see that \eqref{eq:4.2} is equivalent to 
\[
\langle \ddot{\ww}(t), \pp(t) \rangle_\rho - (\ww(t), \ddiv \sigma(\pp(t), \xi(t)))_\Omega  = \langle y(t), \gamma \xi(t) \rangle_\Gamma,
\]
and this is equivalent to \eqref{lem:4.1}, which finishes the proof.
\end{proof}

Note that, implicitly, we have proved that
\begin{equation}\label{lem:4.1real}
\int_0^T (\bff f(t),Sy(t))_\rho\mathrm dt
=\int_0^T \langle R\bff f(t),y(t)\rangle_\Gamma \mathrm dt
\qquad \forall\bff f\in \mathcal U, \quad y\in \mathcal Z.
\end{equation}
Proposition \ref{prop:4.6} implies
that the first order optimality conditions for the control problem \eqref{eq:cont2} 
\begin{equation} \label{eq:contOpt}
\bar{z} \in \mathcal Z_\mathrm{ad} \qquad \qquad \langle j'(\bar{z}), z - \bar{z} \rangle \geq 0 \quad \forall z \in \mathcal Z_\mathrm{ad},
\end{equation}
can be written as the search for $(\uu,\beta, \bar{z}) \in \mathcal U \times \mathcal C^1([0,T];L^2(\Gamma)) \times \mathcal Z_\mathrm{ad}$ satisfying
\begin{alignat*}{4}
&\uu = S\bar{z}, \\
	& \beta = R(\uu - \uu_d), \\
	&\int_0^T \langle \beta(t), z(t) - \bar{z}(t)  \rangle \mathrm d t + \alpha \int_0^T \langle \dot{\bar{z}}(t), \dot{z}(t) - \dot{\bar{z}}(t) \rangle_\Gamma \mathrm d t \geq 0, \qquad \forall z  \in \mathcal Z_\mathrm{ad}. 
\end{alignat*}

\subsection{The semidiscrete model}

Similar to the previous section, we here state some properties and theorems related to the semidiscrete control problem introduced in Section \ref{sec:SCP}. 

\begin{theorem} \label{thm:stateh}
If $z \in \mathcal Z$, then \eqref{eq:SDstate} has a unique solution that satisfies the bounds 
\[
\|\uu_h(t)\|_{1,\Omega} + \|\psi_h(t)\|_{1,\Omega} \lesssim \int_0^t \|z(\tau)\|_\Gamma \mathrm d \tau + \int_0^t \|\dot{z}(\tau) \|_\Gamma \mathrm d \tau.
\] 
Therefore $S_h:\mathcal Z \to\mathcal U$ is uniformly bounded.
\end{theorem}

\begin{proof}
Everything follows as in the proof to Theorem \ref{thm:1} in Appendix \ref{sup:sec:Pfs} 
after defining discrete versions of $M_\Omega, M_\Gamma$ and the divergence operator. The details are very similar to what can be found in \cite{BrSaSa2017}.
\end{proof}

Statements (a)-(d) of  Theorem \ref{prop:wlsc} still hold for $S_h, \mathcal Z_\mathrm{ad}^h$, and $j_h$, as does the conclusion, so with an appropriate change of notation we have the following. 

\begin{theorem}
There exists a unique solution to the semidiscrete control problem 
\begin{equation} \label{eq:semCont}
j_h(z_h) = \min! \qquad z_h \in \mathcal Z_\mathrm{ad}^h.
\end{equation}
\end{theorem}

Using the same notation as with \eqref{eq:SDstate}, we state the semidiscrete version of the adjoint equation \eqref{eq:adj} as well as give a well-posedness result. 
\begin{theorem} \label{thm:adjh}
For $\mathbf f \in \mathcal U$ and every $t \in [0,T]$, the problem
\begin{subequations} \label{eq:SDadj}
\begin{alignat}{3}
& (\pp_h,\xi_h)\in 
\mathcal C^2([0,T];\bff V_h)\times \mathcal C^0([0,T];W_h^G),\\
& (\rho \ddot{\pp}_h(t), \ww)_\Omega + a((\pp_h(t), \xi_h(t)), (\ww, \varphi)) = & (\rho \bff f(t), \ww)_\Omega \quad \forall (\ww, \varphi) \in \bff V_h \times W_h^G,\\
&\pp_h(T) = \bff 0, \qquad \dot{\pp}_h(T) = \bff 0,
\end{alignat}
\end{subequations}
is uniquely solvable and we have the estimate
\begin{alignat*}{4}
\|\pp_h(t)\|_{1,\Omega} + \|\xi_h(t)\|_{1,\Omega} & \lesssim \int_t^T \|\bff f(\tau)\|_\Omega \mathrm d \tau.
\end{alignat*}
Therefore, the operator $R_h:\mathcal U \to \mathcal X$ given by $R_h\mathbf f=\gamma\xi_h$, where $(\pp_h,\xi_h)$ solve \eqref{eq:SDadj}, is uniformly bounded.
\end{theorem}

\begin{proposition} \label{prop:4.6h}
The G\^{a}teaux derivative of $j_h(z)$ in the direction $y\in \mathcal Z$ is given by 
\[
\langle j_h'(z), y\rangle = 
\int_0^T \langle \beta_h(t), y(t)\rangle_\Gamma  \mathrm d t + \alpha \int_0^T \langle \dot{z}(t),\dot y(t)  \rangle_\Gamma \mathrm d t, 
\]
where $\beta_h = R_h(S_h z- \uu_d).$ 
\end{proposition}

\begin{proof}
The proof is similar to the one for Proposition \ref{prop:4.6}. The key step is the transposition formula
\begin{equation}\label{lem:4.1realh}
\int_0^T (\bff f(t),S_hy(t))_\rho\mathrm dt
=\int_0^T \langle R_h\bff f(t),y(t)\rangle_\Gamma \mathrm dt
\qquad \forall\bff f\in \mathcal U, \quad y\in \mathcal Z.
\end{equation}
(compare with \eqref{lem:4.1real}), which can be proved with the same techniques. 
\end{proof}

\noindent It now follows that the semidiscrete optimality conditions consist of
\begin{equation} \label{eq:semOpt}
\bar{z}_h \in \mathcal Z_\mathrm{ad}^h \qquad \qquad  \langle j_h'(\bar{z}_h), z_h - \bar{z}_h \rangle \geq 0, \quad \forall z_h \in \mathcal Z_\mathrm{ad}^h, 
\end{equation}
or equivalently, finding $(\uu_h, \beta_h, \bar{z}_h) \in \mathcal U_h \times \mathcal C^0([0,T], \gamma W_h^G) \times \mathcal Z_\mathrm{ad}^h$ that solve the system
\begin{alignat*}{3}
&\uu_h = S_h \bar{z}_h, \\
&\beta_h = R_h(\uu_h - \uu_d),\\
&\int_0^T \langle \beta_h(t), z_h(t) - \bar{z}_h(t)  \rangle_\Gamma \mathrm d t + \alpha \int_0^T \langle \dot{\bar{z}}_h(t), \dot{z}_h(t) - \dot{\bar{z}}_h(t) \rangle_\Gamma \mathrm d t \geq 0 \qquad \forall z_h  \in \mathcal Z_\mathrm{ad}^h. 
\end{alignat*}
Here we are using the notation $\gamma W_h^G$ to be the space $\{ \gamma \varphi : \varphi \in W_h^G\}$.

\subsection{Gradient and projection}

As part of the needs to apply a projected gradient-type method, we have to introduce the gradient of the functional $j_h$ and the projection operator on the admissible set. We will only deal with them at the semidiscrete level, although all arguments below can be reproduced for the continuous problem.
 
Given $z_h\in \mathcal Z_{\mathrm{ad}}^h$, we consider $g_h\in \mathcal Z_h$ to be the only solution of
\[
\llbracket g_h,y_h\rrbracket_{\mathcal Z} = \langle j'_h(z_h),y_h\rangle\qquad \forall y_h\in \mathcal Z_h,
\]
where
\[
\llbracket g,y\rrbracket_{\mathcal Z} :=\int_0^T \langle \dot g(t),\dot y(t)\rangle_\Gamma \mathrm dt
\]
is the inner product associated to the norm in $\mathcal Z$. 

We finally introduce the best approximation operator $\mathcal Q:\mathcal Z_h \to \mathcal Z^h_{\mathrm{ad}}$ given by the solution of the quadratic problem with linear inequality constraints
\[
\triple{z_h-\mathcal Q z_h}{\mathcal Z}^2=\min!
\qquad \mathcal Qz_h\in \mathcal Z^h_{\mathrm{ad}}.
\]

\section{Convergence and error analysis}\label{sec:conv}

Now that the both the continuous and semidiscrete control problems have been stated and their respective properties explored, we can examine the error due to the semidiscretization in space. 

\subsection{Estimates for Galerkin semidiscretization}

We first examine the error in the approximation of the state and adjoint equations. The analysis is rendered easier if we introduce an elliptic projection associated to the bilinear form $a$. We consider the space $\mathcal M:=\{\mathbf m\in \mathbf H_D^1(\Omega)\,:\,\boldsymbol\varepsilon(\mathbf m)=0\}$.  This finite-dimensional space is: (a) the space of infinitesimal rigid motions (affine displacement fields with skew-symmetric gradient) if $\Gamma_D$ is trivial; (b) zero, otherwise. We assume $\mathcal M\subset \mathbf V_h$, which is an actual hypothesis only when $\Gamma_D$ is trivial. We then consider the orthogonal projection $\mathrm P:\mathbf L^2(\Omega) \to \mathcal M$ and the operator $\Pi:\mathbf H^1_D(\Omega)\times H^1_G(\Omega)\to \mathbf V_h \times W_h^G$ given by $\Pi(\mathbf u,\psi)=(\widehat\uu_h,\widehat\psi_h)$ being the only solution (see Lemma \ref{lemma:quasi}) of
\begin{subequations}\label{eq:proj}
\begin{alignat}{6}
& (\widehat\uu_h,\widehat\psi_h)\in \bff V_h\times W_h^G,\\
& a((\widehat\uu_h,\widehat\psi_h), (\ww, \varphi)) = a((\uu, \psi), (\ww, \varphi)) \qquad \forall (\ww, \varphi) \in \bff V_h \times W_h^G,\\
& \mathrm P \widehat\uu_h = \mathrm P \uu.
\end{alignat}
\end{subequations}
The best approximation operator on the product space $\bff V_h\times W_h^G$ can be decomposed as a pair of independent operators $\mathbf I_h:\HH_D^1(\Omega) \to \bff V_h$ and $I_h:H^1_G(\Omega)\to W_h^G$ satisfying
\[
\| \uu-\mathbf I_h\uu\|_{1,\Omega}
=\min_{\ww\in \bff V_h}\| \uu-\ww\|_{1,\Omega}
\qquad
\|\psi-I_h\psi\|_{1,\Omega}
=\min_{\varphi\in W_h^G} \|\psi-\varphi\|_{1,\Omega}
\]
for arbitrary $\uu$ and $\psi$. 

\begin{lemma}\label{lemma:quasi}
The equations \eqref{eq:proj} are uniquely solvable and, therefore, the projection $\Pi$ is well-defined. Moreover, $\Pi$ is quasioptimal, i.e.,
\[
\| (\uu,\psi)-\Pi(\uu,\psi)\|_{1,\Omega}
\lesssim \| \uu-\mathbf I_h\uu\|_{1,\Omega}+\|\psi-I_h\psi\|_{1,\Omega}.
\]
\end{lemma}

\begin{proof}
Problem \eqref{eq:proj} is equivalent to
\begin{subequations}\label{eq:prf}
\begin{alignat}{6}
& (\widehat\uu_h,\widehat\psi_h)\in \bff V_h\times W_h^G,\\
& a((\uu - \widehat\uu_h, \psi - \widehat\psi_h),(\ww, \varphi)) + 
(\mathrm P ( \uu -  \widehat\uu_h), \ww)_\Omega = 0 \quad \forall (\ww, \varphi) \in \bff V_h \times W_h^G.
\end{alignat}
\end{subequations}
This is a simple consequence of the fact that
\[
a((\mathbf m,0),(\ww,\varphi))=0 \qquad \forall \mathbf m\in \mathcal M,
\quad \forall (\ww,\varphi) \in \bff V_h \times W_h^G,
\]
and that by hypothesis $\mathcal M\times \{0\}\subset \bff V_h \times W_h^G.$  {In $\mathbf H_D^1(\Omega)$ we have the norm equivalence 
\[
\|\eps{\uu}\|_\Omega^2 + \|\mathrm P \uu\|_\Omega^2 \approx \|\uu\|^2_{1,\Omega}.
\]
One direction of the equivalence is a straightforward application of the boundedness of the operators. To see the other direction, we first note that for any $\mathbf u  \in \mathbf H^1(\Omega)$, we have the orthogonal decomposition $\mathbf u = P\mathbf u + (I-P)\mathbf u$.  Furthermore, since $P\mathbf u \in \mathcal M$ we have that $\varepsilon(P\mathbf u) = 0$.  Using this and Korn's first and second inequalities \cite[Chapter 10]{McLean2000}, we have 
\begin{alignat*}{3}
\|\varepsilon(\mathbf u)\|_\Omega^2 + \|P\mathbf u\|_\Omega^2 &= \|\varepsilon((I-P)\mathbf u)\|_\Omega^2 + \left(\|P\uu\|_\Omega^2 + \|\varepsilon(P\uu)\|_\Omega^2\right)\\
&\geq \|(I-P)\uu\|_{1,\Omega}^2 + \|P\uu\|_{1,\Omega}^2\\
&= \|\uu\|_{1,\Omega}^2.
\end{alignat*}
With this, we have that} $(\widehat\uu_h,\widehat\psi_h)$ is the Galerkin approximation of $(\uu,\psi)\in \HH^1_D(\Omega)\times H^1_G(\Omega)$ in the discrete space $\bff V_h\times W_h^G$ with respect to the bounded coercive bilinear form 
\[
a((\uu, \psi),(\ww, \varphi)) + 
(\mathrm P  \uu , \ww)_\Omega.
\]
The result is then a straightforward consequence of C\'ea's lemma.
\end{proof}

\begin{proposition}\label{thm:est}
Let $(\uu, \psi)$ be the solution to \eqref{eq:state} and $(\uu_h, \psi_h)$ its Galerkin approximation \eqref{eq:SDstate}. If $\uu \in \CC^2 ([0,T];\HH_D^1(\Omega))$, then for every $t \in [0,T]$,
\begin{alignat*}{3}
\| \uu (t) - \uu_h(t)\|_{1,\Omega} + \|\psi(t) - \psi_h(t)\|_{1,\Omega} & \lesssim \|\uu(t) - \bff I_h \uu(t) \|_{1,\Omega} + \|\psi(t) - I_h \psi(t)\|_{1,\Omega}\\
&\hspace{-29 pt} + \int_0^t \left(\| \ddot{\uu}(\tau) - \bff I_h \ddot{\uu}(\tau)\|_{1,\Omega} + \|\ddot{\psi}(\tau) - I_h \ddot{\psi}(\tau)\|_{1,\Omega} \right) \mathrm d \tau.
\end{alignat*}
\end{proposition}

\begin{proof}
Note first that if $\uu \in \CC^2 ([0,T];\HH_D^1(\Omega))$, then $\psi\in \mathcal C^2([0,T];H^1_G(\Omega))$, due to the fact that $\psi(t)$ can be computed (for every $t$) from the relation
\[
(\boldsymbol\kappa\nabla\psi(t),\nabla\varphi)_\Omega
=(\boldsymbol\varepsilon(\uu(t)),\mathcal E\nabla\varphi)_\Omega
\qquad\forall\varphi\in H^1_G(\Omega).
\]
In particular we have enough smoothness in the space variable after two time derivatives to write $\langle \ddot\uu(t),\ww\rangle_\rho=(\rho\,\ddot\uu(t),\ww)_\Omega$ for all $t$ and $\ww$. Consider the elliptic projection applied to the continuous solution $(\widehat\uu_h(t),\widehat\psi_h(t)):=\Pi(\uu(t),\psi(t))$. It is clear that
\[
\tfrac{\mathrm d^2}{\mathrm dt^2}(\widehat\uu_h(t),\widehat\psi_h(t))
=\Pi(\ddot\uu(t),\ddot\psi(t)),
\]
and therefore $\Pi(\uu,\psi)=(\widehat\uu_h,\widehat\psi_h)\in \mathcal C^2([0,T];\bff V_h\times W_h^G)$. The discrete pair $(\bff u_h,\psi_h)$ is also in this space, due to the fact that we are working in finite dimensions and the norm in the final space is not relevant for smoothness. Now consider the error quantities
\[
\bff e_u(t):=\widehat\uu_h(t)-\uu_h(t),
\qquad
e_\psi(t):=\widehat\psi_h(t)-\psi_h(t),
\]
and the approximation error $\boldsymbol\varepsilon_u(t):=\widehat\uu_h(t)-\uu(t)$. Therefore {, after plugging $\mathbf e_u(t)$ and $e_\psi(t)$ into \eqref{eq:SDstate}, we obtain}
\begin{subequations}
\begin{alignat}{6}
& (\bff e_u,e_\psi)\in \mathcal C^2([0,T];\bff V_h\times W_h^G),\\
& (\rho\,\ddot{\bff e}_u(t),\ww)_\Omega
	+a((\bff e_u(t),e_\psi(t)),(\ww,\varphi))
	=(\rho\ddot{\boldsymbol\varepsilon}_u(t),\ww)_\Omega
	 \hspace{4 pt}\forall (\ww,\varphi)\in \bff V_h\times W_h^G,\\
& \bff e_u(0)=\dot{\bff e}_u(0)= \mathbf 0,
\end{alignat}
\end{subequations}
as follows from the definition of the elliptic projection $\Pi$ with \eqref{eq:proj}. We can then apply Theorem \ref{thm:adjh} with $\bff f:=\ddot{\boldsymbol\varepsilon}_u(T-\,\cdot\,)$ to obtain bounds for $(\bff e_u(T-\cdot),e_\psi(T-\cdot))$. The rest of the proof follows from a direct application of Lemma \ref{lemma:quasi}.
\end{proof}

At this moment, we start dealing with asymptotic properties. We thus assume that we have collection of subspaces $\{\bff V_h\times W_h^G\}$ directed in a parameter $h\to 0$ such that
\begin{equation}\label{eq:approx}
\bff I_h\uu\longrightarrow \uu, \qquad I_h\psi\to\psi
\qquad \forall(\uu,\psi)\in \HH^1_D(\Omega)\times H^1_G(\Omega),
\end{equation}
where the arrow describes limits as $h\to 0$ in the corresponding spaces

\begin{theorem}\label{thm:Sh2S}
Assuming that \eqref{eq:approx} holds, we have $S_hz\to Sz$ in $\mathcal U$ for all $z\in \mathcal Z.$
\end{theorem}

\begin{proof}
We need to carefully proceed in a series of steps. If we take $(\ww,\varphi)\in \mathcal C^0([0,T];\HH^1_D(\Omega)\times H^1_G(\Omega))$, then the hypothesis above and a compactness argument imply that
\[
\max_{0\le t\le T} \| \bff I_h \ww(t)-\ww(t)\|_{1,\Omega}
+\max_{0\le t\le T}\|I_h\psi(t)-\psi(t)\|_{1,\Omega} \to 0.
\]
Consider now the set
\[
\mathcal Z_{\mathrm{str}}:=\{ z\in \mathcal C^3([0,T];L_0^2(\Gamma))\,:\,
z(0)=\dot z(0)=\ddot z(0)=0\},
\]
and note that if $z\in \mathcal Z_{\mathrm{str}}$, then $\ddot z\in \mathcal Z$. Let then $(\bff v,\eta)$ be the solution to the state equations \eqref{eq:state} when we use $\ddot z$ as data. The pair
\[
(\uu,\psi)(t):=
\int_0^t\left(\int_0^{\tau_1}(\mathbf v(\tau_2),\eta(\tau_2))
\mathrm d\tau_2\right)\mathrm d\tau_1,
\]
is then clearly a solution to \eqref{eq:state} with $z$ as input data. Moreover we have $\ddot{\bff u}=\bff v\in \mathcal C^0([0,T];\HH^1_D(\Omega))$. Using Proposition \ref{thm:est}, it then follows that
\[
\triple{S_hz-Sz}{\mathcal U} \to 0 \qquad \forall z\in \mathcal Z_{\mathrm{str}}.
\]
Finally, the result follows from the density of $\mathcal Z_{\mathrm{str}}$ in $\mathcal Z$ (this can be proved by a standard cut-off and mollification argument), the boundedness of $S:\mathcal Z\to\mathcal U$ (Theorem \ref{thm:1}) and the uniform boundedness of $S_h:\mathcal Z\to \mathcal U$ (Theorem \ref{thm:stateh}). 
\end{proof}

\begin{theorem}\label{thm:Rh2R}
Assuming that \eqref{eq:approx} holds, we have $R_h\mathbf f\to R\mathbf f$ in $\mathcal X$ for all $\mathbf f\in \mathcal U.$
\end{theorem}

\begin{proof}
This proof follows a very similar pattern to the one used in Theorem \ref{thm:Sh2S}. We first need to establish a result like Proposition \ref{thm:est} for the difference $(\pp-\pp_h,\xi-\xi_h)$ corresponding to the solutions of the adjoint problem \eqref{eq:adj} and its Galerkin semidiscretization \eqref{eq:SDadj}. This is easy, due to the fact that the error equations are the same, with final values at $T$ instead of initial values at $0$. To have $\pp\in \mathcal C^2([0,T];\HH^1_D(\Omega))$ as needed for the estimate, it is enough to work with $\bff f$ in the space
\[
\mathcal U_{\mathrm{str}}:=\{ \uu\in \mathcal C^1([0,T];\HH^1_D(\Omega))\,:\,\uu(T)=0\},
\]
which is dense in $\mathcal U$. We thus get convergence $R_h\bff f\to R\bff f$ in $\mathcal X$ for $\bff f\in \mathcal U_{\mathrm{str}}$. Finally, we use the boundedness of $R:\mathcal U\to \mathcal X$ (Theorem \ref{thm:2}) and uniform boundedness of $R_h:\mathcal U\to \mathcal X$ (Theorem \ref{thm:adjh}) to extend the result to arbitrary $\bff f\in \mathcal U$. 
\end{proof}

\subsection{Convergence of the semidiscrete control problem} \label{ssec:contConv}

Before we state our results on the semidiscretization error of the functional, we introduce the orthogonal projection $\Pi_h: L^2(\Gamma) \to \mathcal P_0(\Gamma_h)$.  We note that if $z \in \mathcal Z_\mathrm{ad}$, then $\Pi_h z \in \mathcal Z_\mathrm{ad}^h$. 

\begin{theorem}\label{thm:zh2z}
If $z$ solves \eqref{eq:cont2} and $z_h$ solves \eqref{eq:semCont}, then we can bound the semidiscretization error for the optimal control as 
\begin{equation}
\triple{z-z_h}{\mathcal Z}
\lesssim  \triple{z-\Pi_hz}{\mathcal Z}
+\triple{(S-S_h)z}{\mathcal U}
+\triple{(R-R_h)(Sz-\mathbf u_d)}{\mathcal X}
+\triple{\beta-\Pi_h\beta}{\mathcal X}, 
\end{equation}
where $\beta=R(Sz-\uu_d)$. The hidden constants are independent of $h$ and behave as $1/\alpha$ as $\alpha \to 0$. 
\end{theorem}
\begin{proof}
By the optimality conditions \eqref{eq:contOpt} and \eqref{eq:semOpt} we have 
\[
\langle j'(z), z_h - z \rangle \geq 0,\qquad  \qquad 
\langle j_h'(z_h), \Pi_h z - z_h \rangle \geq 0,
\]
since $z_h \in \mathcal Z_\mathrm{ad}$ and $\Pi_h z \in \mathcal Z_\mathrm{ad}^h$.  Adding these together and using Propositions \ref{prop:4.6} and \ref{prop:4.6h}, we obtain 
\begin{alignat*}{4}
0 \leq &\, \langle j'(z), z_h -z \rangle + \langle j_h'(z_h), \Pi_h z - z_h \rangle\\
= &\int_0^T \langle \beta(t),z_h(t) - z(t) \rangle_\Gamma \mathrm d t + \alpha \int_0^T \langle \dot{z}(t), \dot{z}_h(t) - \dot{z}(t) \rangle_\Gamma \mathrm d t\\
&+ \int_0^T \langle \beta_h(t),\Pi_h z(t) - z_h(t)  \rangle_\Gamma \mathrm d t + \alpha \int_0^T \langle \dot{z}_h(t), \Pi_h \dot{z}(t) - \dot{z}_h (t) \rangle_\Gamma \mathrm d t,
\end{alignat*}
where $\beta_h=R_h(S_hz_h-\uu_d)$. 
Careful manipulation and rearrangement yields the quantity we wish to bound on the left hand side. 
\[
\alpha \int_0^T \|\dot{z}(t) - \dot{z}_h(t) \|_\Gamma^2 \mathrm d t \leq  \, \langle j_h'(z_h), \Pi_h z -z \rangle
+ \int_0^T \langle z_h(t) - z(t), \beta(t) - \beta_h(t) \rangle_\Gamma \mathrm d t.
\]
We can write this as  
\begin{equation} \label{eq:bound0}
\alpha \triple{z-z_h}{\mathcal Z}^2  \leq  \left|\langle j_h'(z_h), \Pi_h z- z\rangle\right| + \int_0^T \langle z_h(t) -z(t),\beta(t) - \beta_h(t) \rangle_\Gamma \mathrm d t,
\end{equation}
and we consider the two terms on the right separately to arrive at a final bound. To simplify some lengthy expressions to come we will use the approximation error
\[
\varepsilon_z(t):=\Pi_hz(t)-z(t),
\]
and note that $\dot\varepsilon_z(t)=\Pi_h\dot z(t)-\dot z(t)$. We also collect some bounds (Theorems \ref{thm:stateh} and \ref{thm:adjh}) in a constant $C_{\mathrm{stb}}>0$ such that
\begin{subequations}\label{eq:constantC}
\begin{equation}
\triple{S_h}{\mathcal Z \to\mathcal U}
+\triple{R_h}{\mathcal U\to\mathcal X}
+\triple{R_hS_h}{\mathcal Z \to \mathcal X}
\le C_{\mathrm{stb}}\qquad \forall h,
\end{equation}
and consider the constant
\begin{equation}
C_{\mathrm{Pnc}}:=\sup_{0\neq z\in \mathcal Z}\frac{\triple{z}{\mathcal X}}{\triple{z}{\mathcal Z}},
\end{equation}
\end{subequations}
for the Poincar\'e-like inequality bounding the norm of $\mathcal X$ by the norm of $\mathcal Z$.

We  {begin by once again recalling the characterization of the G\^ateaux derivative in Proposition \ref{prop:4.6h} and then} adding and subtracting  {
\begin{alignat*}{4}
\int_0^T \langle R_h(S_h z - \mathbf u_d)(t),\varepsilon_z(t)\rangle_\Gamma \mathrm dt, \qquad \qquad & \int_0^T \langle R_h(S z - \mathbf u_d)(t),\varepsilon_z(t)\rangle_\Gamma \mathrm dt,\\
\int_0^T \langle R(S z - \mathbf u_d)(t),\varepsilon_z(t)\rangle_\Gamma \mathrm dt, \qquad \qquad &\alpha \int_0^T \langle \dot{z}(t), \dot{\varepsilon}_z(t)\rangle_\Gamma \mathrm d t, 
\end{alignat*}
as well as adding 
\[
-\int_0^T\langle \Pi_h \beta(t), \varepsilon_z(t)\rangle_\Gamma \mathrm dt \qquad \text{and} \qquad -\alpha\int_0 \langle \Pi_h \dot{z}(t), \dot{\varepsilon}_z(t)\rangle_\Gamma \mathrm dt, 
\]
which are both zero due to the orthogonal projection $\Pi_h$, we obtain} (recall that $\beta=R(Sz-\uu_d)$ and $\beta_h=R_h(S_hz-\uu_d)$)
\begin{alignat*}{4}
\langle j_h'(z_h), \varepsilon_z \rangle = 
& \int_0^T \langle \varepsilon_z(t),
R_hS_h(z_h-z)(t) \rangle_\Gamma \mathrm d t  + \int_0^T \langle \varepsilon_z(t), R_h(S_h-S)z(t)\rangle_\Gamma \mathrm d t\\
& + \int_0^T \langle \varepsilon_z(t),
	(R_h-R)(Sz-\uu_d)(t)\rangle_\Gamma \mathrm d t + \int_0^T \langle \varepsilon_z(t),
\beta(t) - \Pi_h \beta(t)\rangle_\Gamma \mathrm d t\\
& + \alpha \int_0^T \langle \dot\varepsilon_z(t),\dot{z}_h(t) - \dot{z}(t)\rangle_\Gamma \mathrm d t 
- \alpha \int_0^T \|\dot\varepsilon_z(t) \|_\Gamma^2 \mathrm d t.
\end{alignat*}
We now apply the Cauchy-Schwarz inequality several times in the spaces $\mathcal X$ and $\mathcal Z$, boundedness estimates collected in \eqref{eq:constantC}, and Young's inequality, to estimate
\begin{alignat}{3}
\nonumber
\left|\langle j_h'(z_h), \varepsilon_z \rangle \right| 
\le \, & \frac\alpha4 \triple{z-z_h}{\mathcal Z}^2
	+ \frac{C_{\mathrm{stb}}^2}\alpha\triple{\varepsilon_z}{\mathcal X}^2 
	+\frac32\triple{\varepsilon_z}{\mathcal X}^2 \\
\nonumber
		& \hspace{-9pt} +\frac12\left( C_{\mathrm{stb}}^2\triple{(S_h-S)z}{\mathcal U}^2
				+ \triple{(R_h-R)(Sz-\uu_d)}{\mathcal X}^2
				+ \triple{\beta-\Pi_h\beta}{\mathcal X}^2 \right)\\
\label{eq:bound1}
& \hspace{-9pt} + \frac\alpha4 \triple{z-z_h}{\mathcal Z}^2 
+ 2\alpha  \triple{\varepsilon_z}{\mathcal Z}^2.
\end{alignat}
Turning our attention to the second quantity in \eqref{eq:bound0}, we add and subtract inner products similar to what we have done above to eliminate a non-positive term
\begin{multline*}
\int_0^T\langle z_h(t) - z(t),\beta(t) - \beta_h(t)\rangle_\Gamma \mathrm d t  \\ =\int_0^T \langle z_h(t)-z(t),
(R-R_h)(Sz-\uu_d)(t)+R_h(S-S_h)z(t)+R_hS_h(z-z_h)(t)\rangle_\Gamma \mathrm d t. 
\end{multline*}
By \eqref{lem:4.1realh}, we have
\[
\int_0^T \langle z_h(t) - z(t),R_hS_h(z-z_h)(t)\rangle_\Gamma \mathrm d t
=\int_0^T (S_h(z_h-z)(t),S_h(z-z_h)(t))_\rho\mathrm dt\le 0.
\]
Therefore, by Young's inequality and \eqref{eq:constantC}
\begin{alignat}{6}
\label{eq:bound2}
\int_0^T\langle z_h(t) - z(t),\beta(t) - \beta_h(t)\rangle_\Gamma \mathrm d t
\le &
\frac\alpha4 \triple{z-z_h}{\mathcal Z}^2 \\
\nonumber
&\hspace{-36pt} +\frac{2C_{\mathrm{Pnc}^2}}\alpha 
\left(\triple{(R-R_h)(Sz-\uu_d)}{\mathcal X}^2
+C_{\mathrm{stb}}^2 \triple{(S-S_h)z}{\mathcal U}^2\right).
\end{alignat}
Combining \eqref{eq:bound0}, \eqref{eq:bound1}, and \eqref{eq:bound2}, we have
\begin{alignat*}{6}
\frac\alpha4 \triple{z-z_h}{\mathcal Z}^2
	\le & \left( 2\alpha 	
	+C_{\mathrm{Pnc}}^2\left(\frac32 +\frac{C_{\mathrm{stb}}^2}\alpha\right)\right)
		\triple{\varepsilon_z}{\mathcal Z}^2
			+\frac12\triple{\beta-\Pi_h\beta}{\mathcal X}^2\\
	&\hspace{-33pt}+C_{\mathrm{stab}}^2
	\left(\frac12+\frac{2C_{\mathrm{Pnc}}^2}\alpha\right)
	\triple{(S-S_h)z}{\mathcal U}^2
		+\left(\frac12+\frac{2C_{\mathrm{Pnc}}^2}\alpha\right)	
		\triple{(R-R_h)(Sz-\uu_d)}{\mathcal X}^2,
\end{alignat*}
from where the result follows.
\end{proof}

\begin{corollary}
If we assume that $(\mathbf I_h \mathbf w, I_h \varphi, \Pi_h \eta)  \to (\mathbf w, \varphi, \eta) \in \HH_D^1(\Omega) \times  H_G^1(\Omega) \times L^2(\Gamma)$ for all $(\mathbf w,\varphi,\eta)\in \HH_D(\Omega)\times H_G^1(\Omega)\times L^2(\Gamma)$, then the semidiscrete control $z_h$ converges to the continuous control $z$ in $\mathcal Z$ and therefore $\mathbf u_h=Sz_h \to \mathbf u=Sz$ in $\mathcal U$.
\end{corollary}

\begin{proof}
Note first that we have Theorems \ref{thm:Sh2S} and \ref{thm:Rh2R} guaranteeing that the middle two terms in the right hand side of the inequality of Theorem \ref{thm:zh2z} converge to zero. Using a compactness argument, it is simple to show that $\Pi_h y(t) \to y(t)$ uniformly in $t$ for every $y\in \mathcal C([0,T];L^2(\Gamma))$. Therefore, since $\Pi_h:\mathcal X \to\mathcal X$ is uniformly bounded, a density arguments shows that $\Pi_h\beta\to \beta$ in $\mathcal X$ for every $\beta\in \mathcal X$. A similar argument can be shown to prove that $\Pi_h z\to z$ in $\mathcal Z$ for arbitrary $z\in \mathcal Z$, which finishes the proof.
\end{proof}

\section{Numerical experiments}\label{sec:num}

Here we present some numerical experiments of the types of problems covered by the theory above.  We will begin with how we carry out the computations.  To verify that the code is computing things properly we show some convergence studies of the discretized state/adjoint solution operator (they are the same modulo data), and the semidiscrete G\^{a}teaux derivative.  This is followed by the some examples showing evidence of the convergence of the discretized optimal control.  We finish by giving snapshots of a simulation.

\subsection{A fully discrete scheme}\label{sec:5.1}

For everything that follows, we take $\Omega$ to be a polyhedral domain that is partitioned in a conforming tetrahedral mesh $\mathcal T_h$. The Finite Element space $W_h$ is the space of globally continuous functions that are polynomials of degree $k$ on each element, i.e., we define $W_h$ exactly as in Section \ref{sec:SCP}.  Additionally we take $\mathbf V_h = W_h^3$.  Unless otherwise stated, all of the experiments use $k\ge 2$, as $k=1$ is known to under-perform in elasticity simulations, even for reasonably well-behaved material properties.  We make use of high-order Gauss-Jacobi quadrature rules to evaluate the integrals in our finite element method so that the approximation error due to the non-constant coefficients does not have an effect on the overall convergence rates.  For the space discretization of the control we take the space $\mathcal P_0(\Gamma)$, of piecewise constant functions on $\Gamma_h$, where $\Gamma_h$ is the partition of the boundary inherited from $\mathcal T_h$. We will also need the subspaces
\[
\mathcal P_0(\Gamma_h)\cap L_0^2(\Gamma)
	\qquad \mbox{and}\qquad
\mathcal P_0^{\mathrm{ad}}(\Gamma_h)
	:=\{ \eta\in \mathcal P_0(\Gamma_h)\cap L_0^2(\Gamma)\,:\, a\le \eta\le b\}.
\]

To discretize the time interval $[0,T]$ we take a partition $t_0 = 0< t_1 < \dots < t_N = T$ with uniform time step $\delta_t := t_n - t_{n-1} = T/N$. Given a function space $X$, we will consider the space of $X$-valued, continuous, piecewise linear functions
\begin{alignat*}{6}
\mathcal P_1^{\mathrm{cont}}(I_N;X):= &
\{ f\in \mathcal C([0,T];X)\,:\, 
f \big|_{(t_{n-1}, t_n)}  \in \mathcal P_1((t_{n-1},t_n);X) \quad n = 1, \dots N\}\\
\subset & H^1(0,T;X).
\end{alignat*}
As a fully discrete space for the control variable we take
\[
\mathcal Z_{\mathrm{fd}}:=
\{ z\in \mathcal P_1^{\mathrm{cont}}(I_N;\mathcal P_0(\Gamma_h)\cap L_0^2(\Gamma))\,:\, z(0)=0\}
\subset \mathcal Z. 
\]
An element $z\in \mathcal Z_{\mathrm{fd}}$ is fully determined by its values $z_n:=z(t_n)\in \mathcal P_0(\Gamma_h)\cap L_0^2(\Gamma)$ (for $n=1,\ldots,N$) and its time derivative is piecewise constant
\[
\dot z|_{(t_{n-1},t_n)}\equiv \dot z_n:= \tfrac1{\delta_t} (z_n-z_{n-1}) \qquad n=1,\ldots,N.
\]

The forward operator is approximated using the Crank-Nicolson method, thus determining, by an implicit unconditionally stable second order in time method, values $(\mathbf u_n,\psi_n)\in \mathbf V_h\times W_h$. Note that this method only uses the time values $z_n=z(t_n)$ of the discrete control. We approximate the functional $j(z)$ by
\begin{alignat*}{3}
j_\mathrm{fd}(z) := & \frac{\delta_t}{12} \sum_{n=1}^N \Big(\|\mathbf u_{n-1} - \mathbf u^d_{n-1}\|_\rho^2 + 4\left\|\tfrac12(\mathbf u_{n-1}+ \mathbf u_{n}) - \mathbf u^d_{n-1/2}\right\|_\rho^2 + \|\mathbf u_{n} - \mathbf u^d_n\|_\rho^2\Big)  \\
& + \frac{\alpha \delta_t}{2} \sum_{n=1}^N \|\dot z_n\|_\Gamma^2,
\end{alignat*}
where $\mathbf u^d_n$ is the weighted $L^2$ projection of $\mathbf u_d(t_n)$ onto $\mathbf V_h$ and $\mathbf u^d_{n-1/2}:=\frac12 (\mathbf u^d_{n-1}+\mathbf u^d_n)$. Note that the penalization term in the functional is computed exactly, while for the term associated to the desired state we build a function in $\mathcal P_1^{\mathrm{cont}}(I_N;\mathbf V_h)$ using the values $\mathbf u_n-\mathbf u_n^d$, and we then integrate exactly in time.

For the adjoint problem, we apply the Crank-Nicolson scheme again (note that this method only uses $\mathbf u_n-\mathbf u^d_n$), outputting time values $(\mathbf p_n,\xi_n)\in \mathbf V_h\times W_h$. The only part of the output that is needed is the trace $\beta_n:=\gamma \xi_n$. A fully discrete version of the gradient is then computed as follows: given $z\in \mathcal Z_{\mathrm{fd}}\cap \mathcal Z_{\mathrm{ad}}$, we look for $g\in \mathcal Z_{\mathrm{fd}}$ such that
\begin{alignat}{6}
\nonumber
\delta_t\sum_{n=1}^N \langle \dot g_n,\dot y_n\rangle_\Gamma
=& \delta_t \sum_{n=1}^{N-1}
\langle \tfrac16 \beta_{n-1}+\tfrac23 \beta_n+\tfrac16 \beta_{n+1}, y_n\rangle_\Gamma +
\delta_t \langle \tfrac16 \beta_{N-1}+\tfrac13\beta_N,y_N\rangle_\Gamma \\
\label{eq:5.100}
& +  \alpha \delta_t \sum_{n=1}^N \left\langle\dot z_n ,\dot y_n \right \rangle_\Gamma
\quad\forall y\in \mathcal Z_{\mathrm{fd}}.
\end{alignat}
The left-hand side of the above equation is the inner product $\llbracket g,y\rrbracket_{\mathcal Z}$. In the right-hand side we have built $\beta\in \mathcal P_1^{\mathrm{cont}}(I_N;\gamma W_h)$ by interpolating the values $\beta_n$ and then we have computed  the resulting integral (note that $y$ is piecewise linear in time too), while the integral associated to the penalization term is computed exactly. There is an easy computational trick to calculate $g$. In a first step, we extend the space $\mathcal Z_{\mathrm{fd}}$ to $\mathcal Z_{\mathrm{fd}}^\star:=\{ z\in\mathcal P_1^{\mathrm{cont}}(I_N;\mathcal P_0(\Gamma_h))\,:\, z(0)=0\}$, i.e., we eliminate the zero average condition in space. Solving for $g\in \mathcal Z_{\mathrm{fd}}^\star$  satisfying equations \eqref{eq:5.100} for all $y\in \mathcal Z_{\mathrm{fd}}^\star$ is equivalent to
solving a very sparse well-conditioned (block-tridiagonal with diagonal blocks) system. As a postprocess, we subtract the average on $\Gamma$ at each time step. This provides the gradient that we wanted to compute. 

To minimize the functional, we use a projected BFGS method using code modified from C.T. Kelley \cite[Chapter 4]{Kelley1999}.  We are using mesh-independent methods because we are taking into account the $H^1$ topology in time.  The projection that we use $\mathcal Q:\mathcal Z_{\mathrm{fd}}\to \mathcal Z_{\mathrm{fd}}\cap \mathcal Z_{\mathrm{ad}}$ can be computed as follows: given $z\in \mathcal Z_{\mathrm{fd}}$ with time values $\{z_n\}$ we minimize the quadratic functional 
\[
\delta_t^{-1}\sum_{n=1}^N \| (z_n-z_{n-1})-(q_n-q_{n-1})\|_\Gamma^2=\triple{ z-q}{\mathcal Z}^2,
\]
looking for time values $\{q_n\}$ in $\mathcal P_0(\Gamma_n)$ satisfying the
restrictions
\[
\int_\Gamma q_n=0, \qquad a\le q_n\le b \qquad n=1,\ldots,N,
\]
i.e., $q_n\in \mathcal P_0^{\mathrm{ad}}(\Gamma_h)$, 
so that the associated $q$ is an element of $\mathcal Z_{\mathrm{fd}}\cap \mathcal Z_{\mathrm{ad}}$. This is a quadratic functional associated to a block-tridiagonal matrix (one block per time step) with diagonal blocks (we are using piecewise constant functions) with linear restrictions. Similar $H^1$ projections have been used in the two recent papers \cite{AnNoVe2018a, AnNoVe2018b}.

\subsection{Code verification}

The next two experiments will serve to show that our code is computing what we expect.  For both experiments our domain $\Omega$ will be the unit cube $(0,1)^3$ and we will use $\mathbf x := (x,y,z)$ to represent points in the domain.  For the Dirichlet part of the boundary $\Gamma_D$ we will take the intersection of the boundary of $\Omega$ with the coordinate planes, i.e.,
\[
\Gamma_D = \{ \mathbf x \in \partial \Omega :  xyz = 0\}.
\]
We use a sequence of meshes on $\Omega$ where we divide $\Omega$ into $M^3$ (for $M\ge 1$) equal cubes with each cube divided into six tetrahedra.  We will take for a mesh parameter $h:= 1/M$.  This means that not all of the meshes in our sequence are nested.  In time, we fix an initial number of equally spaced timesteps, $N_0$, and subsequently for each refinement take $M N_0$ equal time steps to reach $T$, which we take to be 1.  For the mass density in the cube, we use $\rho(\mathbf x) = 1 + |x| + |y|$.  

In the first experiment, we take the dielectric tensor to be a constant matrix 
\[
\boldsymbol{\kappa} = \begin{pmatrix}
19 & 8 & 7\\
8 & 19 & 5\\
7 & 5 & 17
\end{pmatrix}.
\]
We adopt Voigt's notation to replace symmetric indices
\[
(1,1) \leftrightarrow 1 \quad (2,2) \leftrightarrow 2 \quad (3,3) \leftrightarrow 3 \quad (2,3) \leftrightarrow 4 \quad (1,3) \leftrightarrow 5 \quad (1,2) \leftrightarrow 6,
\]
which allows use to formally write the piezolectric tensor as a $6 \times 3$ matrix (even though we write it here transposed for space), where for these experiments we use the constants
\[
\mathcal E = \begin{pmatrix}
2 & 2 & 3 & 5 & 2 & 3\\
1 & 2 & 6 & 3 & 2 & 1\\
4 & 1 & 3 & 3 & 1 & 3
\end{pmatrix}^\top.
\]
For the elastic part of the stress, we use the relationship for a non-homogeneous isotropic material 
\[
\mathcal C \varepsilon(\mathbf u) = 2 \mu \varepsilon(\mathbf u) + \lambda \nabla \cdot \mathbf u I, 
\]
where $I$ is the $3 \times 3$ identity matrix and the Lam\'{e} parameters $\lambda$ and $\mu$ are given by 
\[
\lambda(\mathbf x) =  1 + \frac{1}{1+ |\mathbf x|^2}, \qquad 
\mu(\mathbf x) = 3 + \cos(xyz).
\]
 {All of the above tensors have been chosen for analytical considerations and may not reflect the properties of a physical material.  With these choices, our goal is to setup a benchmark problem to show that our code has been written correctly.}

To test the state equation solver,  {we use the parameters and tensors defined as above and approximate the solution to  
\begin{alignat*}{5}
\rho \ddot{\mathbf u}(t) & = \mathrm{div} \, (\mathcal C \varepsilon(\mathbf u)(t) + \mathcal E \nabla \psi(t)) + \mathbf f(t),\\
0 &= \nabla \cdot (\mathcal E^\top \varepsilon(\mathbf u)(t)  - \kappa \nabla \psi(t)) + f(t),\\
\gamma_D \mathbf u(t) & = \mathbf g_D(t),\\
\gamma_N (\mathcal C \varepsilon(\mathbf u)(t) + \mathcal E \nabla \psi(t)) &= \mathbf g_N(t),\\
G\psi(t) &=0,\\
(\mathcal E^\top \varepsilon(\mathbf u)(t)  - \kappa \nabla \psi(t)) \cdot \boldsymbol{\nu} &= z(t), \\
\mathbf u(0) = \dot{\mathbf u}(0) &= \mathbf 0, 
\end{alignat*}
using the numerical scheme described in Section \ref{sec:5.1}.  The source terms $\mathbf f, f$ and the boundary data $\mathbf g_D, g_N, z$, are} 
defined so that the exact solution to the system is 
\begin{alignat*}{3}
\mathbf u(\mathbf x, t) &= \begin{pmatrix}
H(2t-2/5) \cos(\pi x) \sin(\pi y) \cos(\pi z)\\
H(2t-2/5)(5x^2yz + 4xy^2z + 3xyz^2 + 17)\\
H(2t-2/5) \cos(2x) \cos(3y) \cos(z) \end{pmatrix}, \\
\psi(\mathbf x,t) &= t^2 \left(x^3 + x^3y - 3xy^2z - \frac13 z^3 - \frac1{24}\right),
\end{alignat*}
where $H(t)$ is the polynomial approximation for the Heaviside function
\[
H(t) = \begin{cases}
0 & t\le 0\\
t^5 ( 1-5(t-1) + 15(t-1)^2 - 35(t-1)^3 \\
\hspace{2.8 cm} + 70(t-1)^4 - 126(t-1)^5)  & 0<t<1\\
1 & t\ge 1
\end{cases}.
\]

In Figure \ref{fig:2}, we show the  $L^2(\Omega)$ norm and $H^1(\Omega)$ seminorm of the difference between the exact solution and finite element approximation at the final time using polynomial degree $k=2$ to show the convergence in space. Since the Finite Element method is of higher order than the Crank-Nicolson rule, we expect to see $O(h^2)$ error, however we are refining in both space and time in order to see the expected convergence in space.   \begin{figure}[ht]\center{
\includegraphics[width=0.6\textwidth]{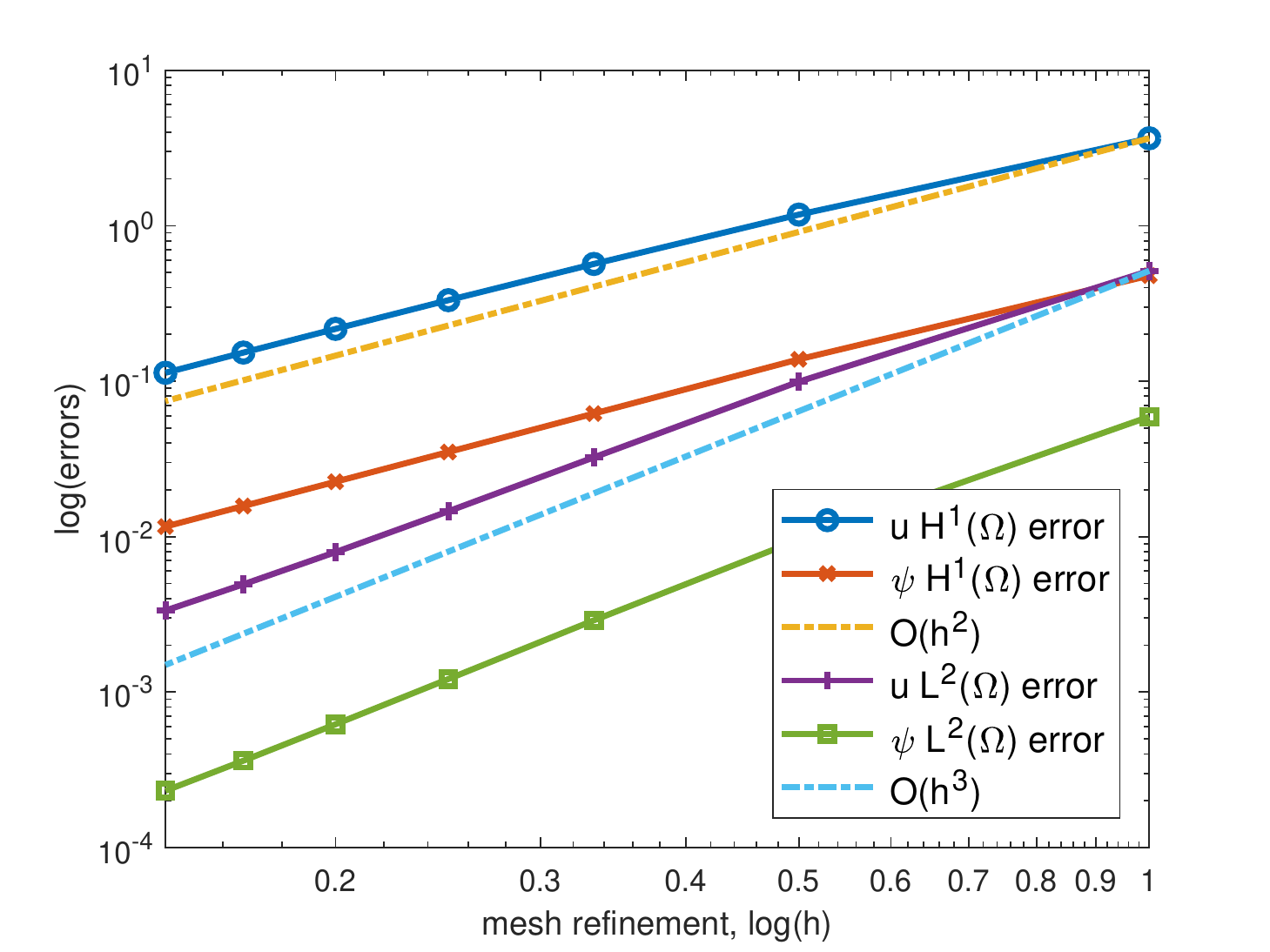}}
\caption{{\footnotesize The $L^2(\Omega)$ and $H^1(\Omega)$ error of the finite element solutions to the state equation compared with exact solutions with refinements in time and space.}}\label{fig:2}
\end{figure}


\subsection{Convergence of the optimal control}

Due to the complexity of the state equation, it is difficult to manufacture an exact solution for the optimal control.  Nevertheless, we would like to know that the control we compute is convergent, matching the theory we have presented in Section \ref{ssec:contConv}.   To achieve that goal, we present some experiments that show evidence that the computed optimal control is converging. 

For what follows, we keep $\Omega, \rho, \lambda, \mu, k, T, \mathcal C, \boldsymbol{\kappa}$, and $\mathcal E$ as in the previous section.  We now take $\Gamma_D$ to be the faces of the cube that intersect with the planes $y=1$ and $y=0$.  {For this and all subsequent experiments we use $\alpha = 10^{-4}$ in the functional.  While we have not conducted a parameter study on $\alpha$ to see what the effect that varying this parameter would have on our solution, we know that we cannot take $\alpha$ too large (as is common in optimal control problems) since this would imply that we are not enforcing adequate control.} Additionally, we take an initial of value of zero for $z_h$ (in space and time) and define all of the components of the desired state by  
\[
t^2 y (y-1)(x+y+z).
\]
Running the projected BFGS optimization routine, we compute the value of the functional (as described above) and the norm of the fully discrete optimal control, 
\[
\triple{z_h}{\mathcal Z} \approx \zeta_h := \delta_t \sum_{n=1}^N \left\|\dot{z}_n\right\|_\Gamma^2.
\]
We refine in both space and time (in the same fashion as the previous experiments) up to $h = 1/8$ and note that the optimization routine converges in the same number of iterations $it$ for each mesh, with the exception of the first mesh which only contains six elements.  This is summarized in Figure \ref{fig:itTable}  and provides evidence that the optimization routine is mesh independent. 
\begin{figure}[ht]\center{
\begin{tabular}{|c||c|c|c|c|c|c|c|c|}
\hline
$h$ & 1 & 1/2 & 1/3 & 1/4 & 1/5 & 1/6 & 1/7 & 1/8\\
\hline
$it$ & 2 & 4 & 4 & 4 & 4 & 4 & 4 & 4\\
\hline
\end{tabular}}
\caption{{\footnotesize The number of iterations needed for convergence in the projected BFGS optimization routine.}} \label{fig:itTable}
\end{figure}

 To show convergence, we compute 
\[
\epsilon_z(h) = \left| \frac{\zeta_h - \zeta_{1/8}}{\zeta_{1/8}}\right|, 
\qquad \epsilon_j(h) = \left|\frac{j_\mathrm{fd}(z_h) - j_\mathrm{fd}(z_{1/8})}{j_\mathrm{fd}(z_{1/8})}\right|, 
\]
for $h \in \{1, 1/2, 1/3, 1/4, 1/5, 1/6, 1/7\}$. The results are shown in 
Figure \ref{fig:3} where we see similar convergence behavior for both the functional and the optimal control. 
\begin{figure}[ht]\center{
\includegraphics[width = 0.6 \textwidth]{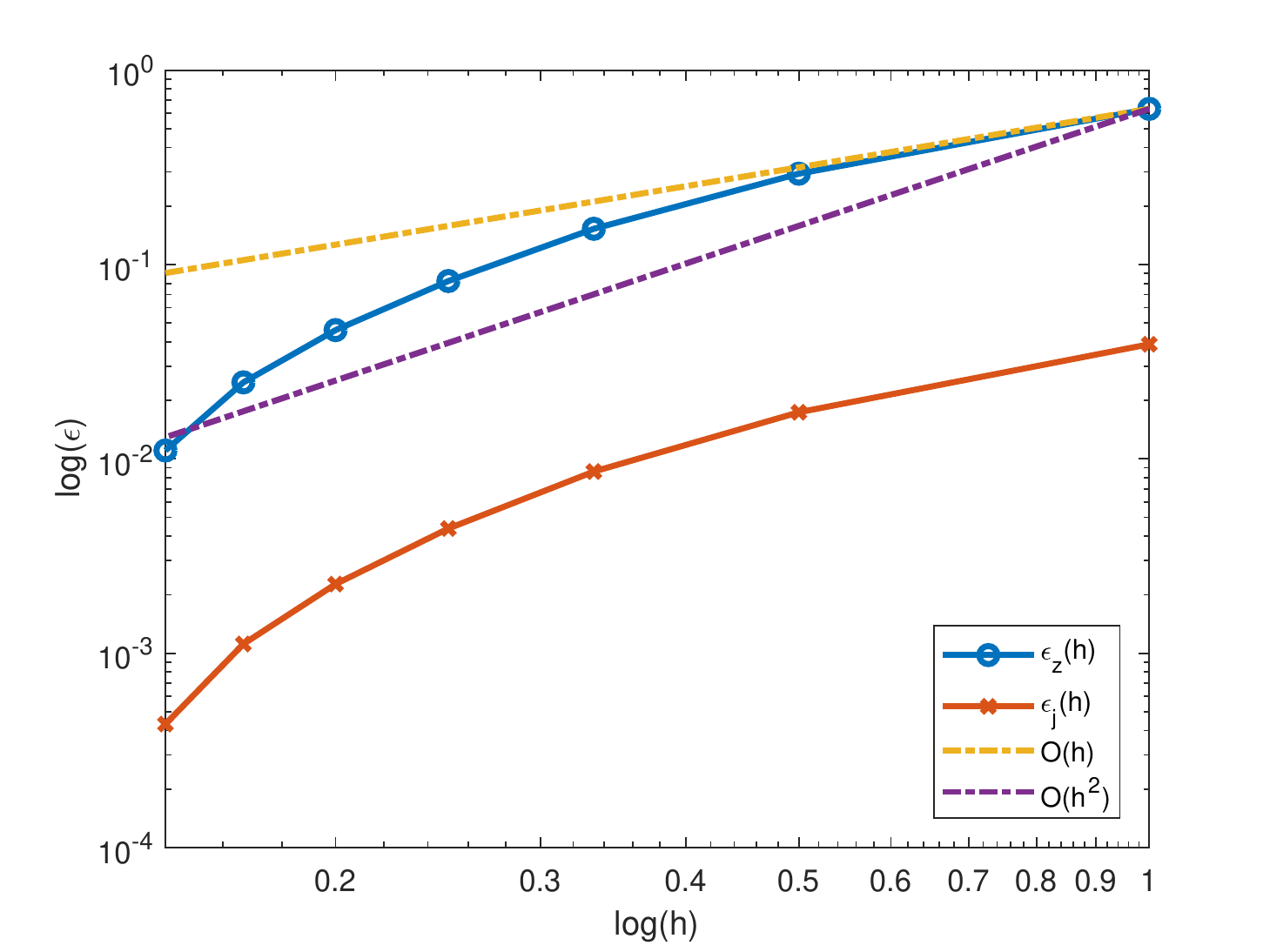} 
}
\caption{{\footnotesize A log log plot in space of $\epsilon_z$ and $\epsilon_j$ compared to convergence lines of order 1 and 2.}} \label{fig:3}
\end{figure}

As more evidence of the convergence of the optimal control, for each of the eight meshes, we compute the integral of the control over each face of the unit cube.  That is, 
\[
\int_{\Gamma_i} z_h\; \mathrm d \Gamma_i \approx \sum_{F \in \mathcal F_i} |F| z_h\big|_F  \qquad  \text{for} \:\:  i = 1, \dots, 6,
\]
where each $\Gamma_i$ represents one of the faces of the cube, and $\mathcal F_i = \Gamma_h \cap \Gamma_i$.  
We plot these integrals as functions of time for each of the space-time refinements over the faces of the cube in Figure \ref{fig:3b}  {including a legend that applies to all six plots}, and see that the plots approach the same values as $h$ decreases.  
\begin{figure}[ht]\center{
\includegraphics[width = \textwidth]{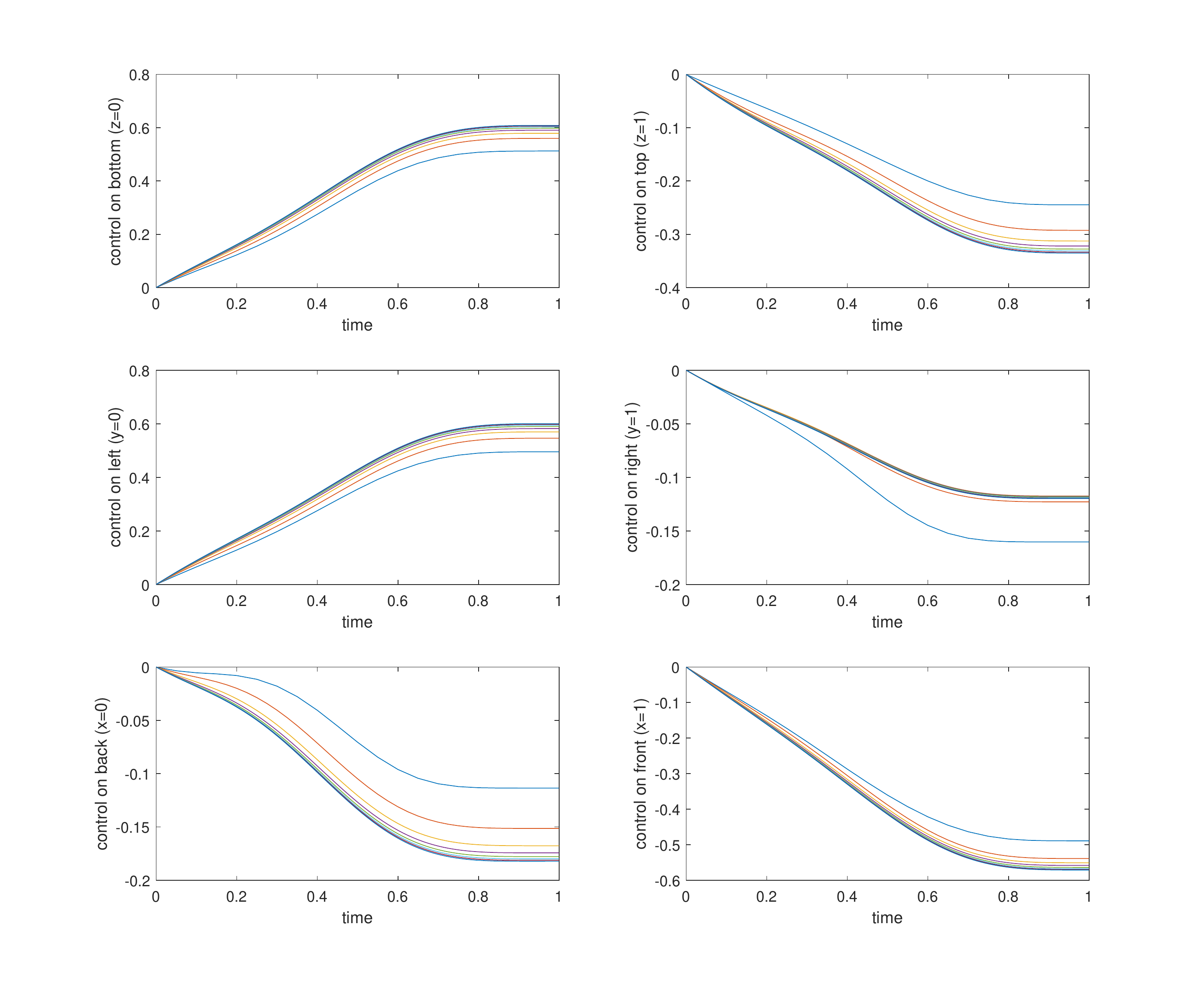}\\
\includegraphics[width = 0.2\textwidth]{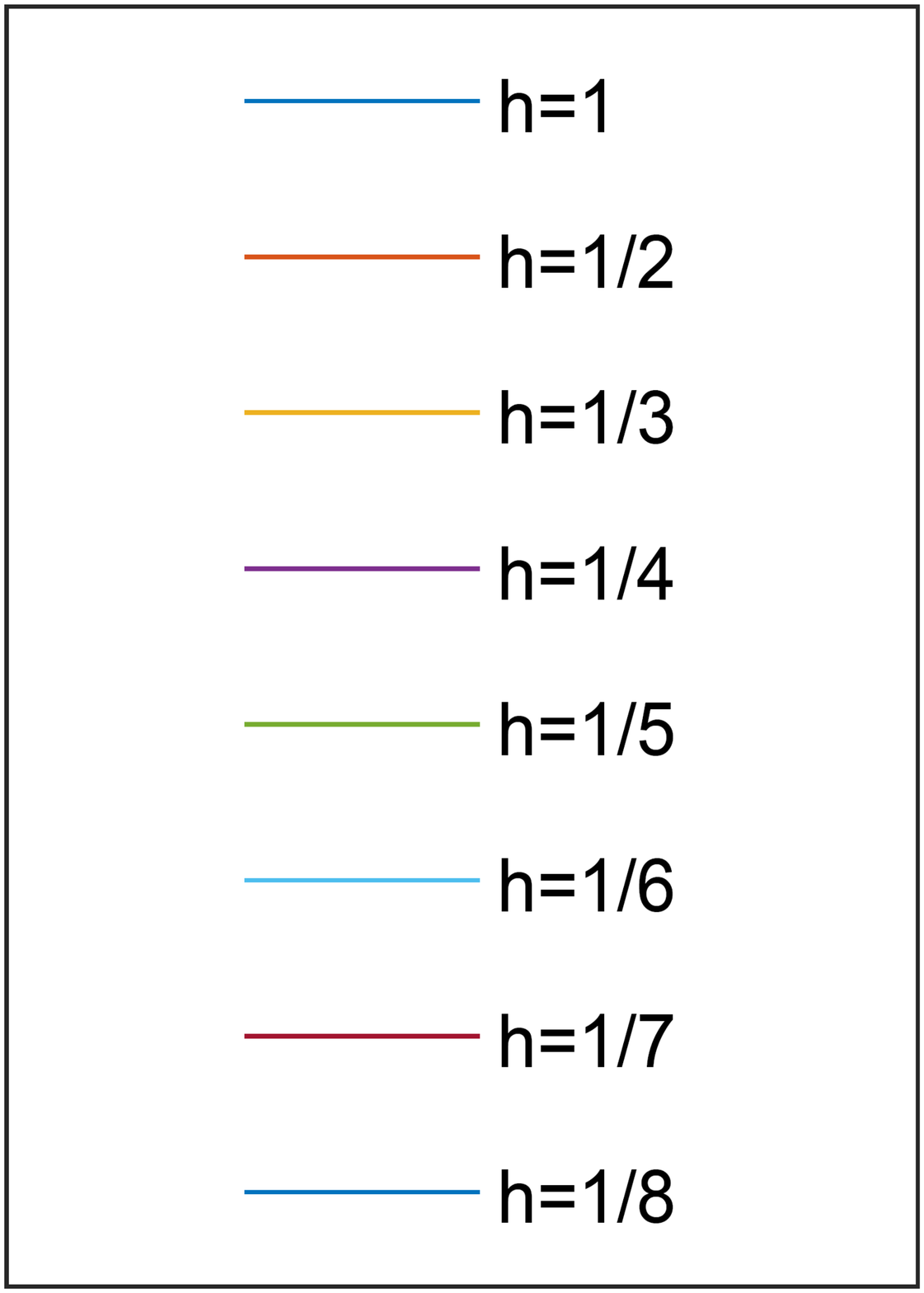}
}
\caption{{\footnotesize The plot shows the computed optimal control for 8 different refinements of the unit cube, integrated over the faces of the cube and plotted as functions of time.}} \label{fig:3b}
\end{figure}

\subsection{Simulation}

In this final section concerning numerical experiments, we describe a simulation in which we show how the optimal control is used to control the deformation of the piezoelectric solid.  To accomplish this task, we again use the unit cube as $\Omega$, this time choosing $\Gamma_D = \Gamma \cap \{z =0\}$, and keep all of the material properties as in the previous experiments.  We use homogeneous boundary and source data  {and take} zero as an initial control.  Using the same polynomial approximation for the Heaviside function $H$ as before we define the window functions 
\[
T_1(t) = H\left(2t - 2/5\right),\qquad \qquad 
T_2(t) = H(t-1/5) H(27/10 - t).
\]
With these functions, we define the desired state
\[
\uu_d = \begin{pmatrix}
T_1(t) (1/2 -y) z\\
T_1(t) (x-1/2) z\\
T_2(t) 2z
\end{pmatrix},
\]
which causes the cube to twist 90 degrees while keeping the bottom face fixed, as well as stretch and compress once in the  {vertical ($z$-axis)} direction.  This choice for the desired state may take us out of the realm of ``small deformations," but we choose it so that we can have something substantial to compare our simulation to.  For space discretization we partition the unit cube into 64 smaller cubes, and each of those into 6 tetrahedra, while in time we take 401 timesteps equally spaced by timestep $\delta = 0.0125$.  {We solve for the optimal control $z_h$.  This quantity is then used as Neumann boundary data for the state equation, where again the Dirichlet boundary (where we implement homogeneous boundary conditions) is the surface of the cube that intersects the plane $z=0$ (bottom face), and the Neumann boundary comprises the remaining 5 surfaces of the cube.  We then solve the state equation, with this data,  using $\mathcal P_3$ finite elements.}  In Figure \ref{fig:5} we show several snapshots from the simulation, showing the computed solution $\uu_h$ on the left and the desired state $\uu_d$ on the right.  The color on both figures is the value of the control. 
\begin{figure}[ht]\center{
\includegraphics[width=0.6\textwidth]{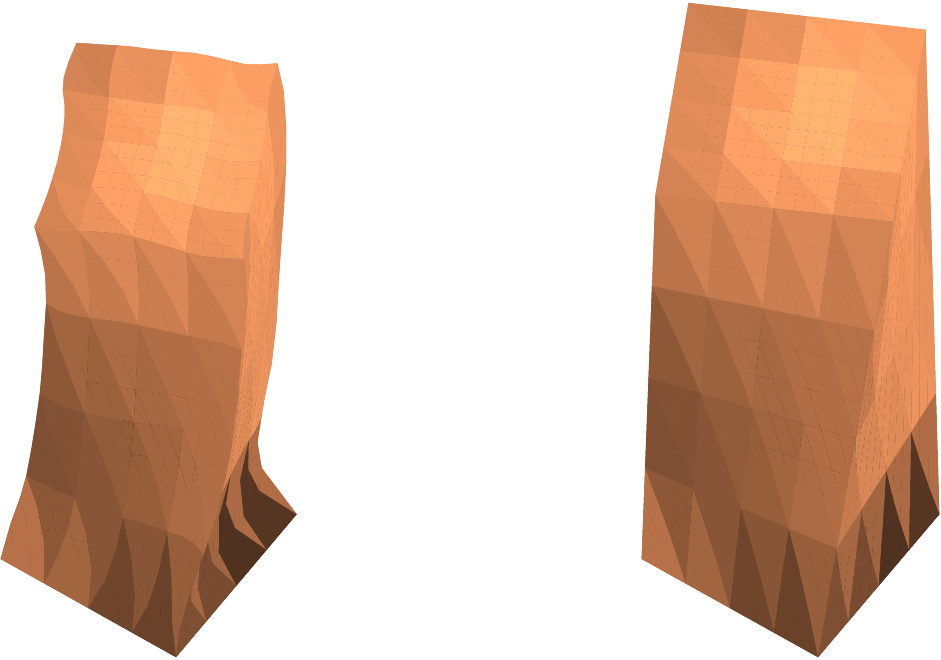}\\
\includegraphics[width=0.6\textwidth]{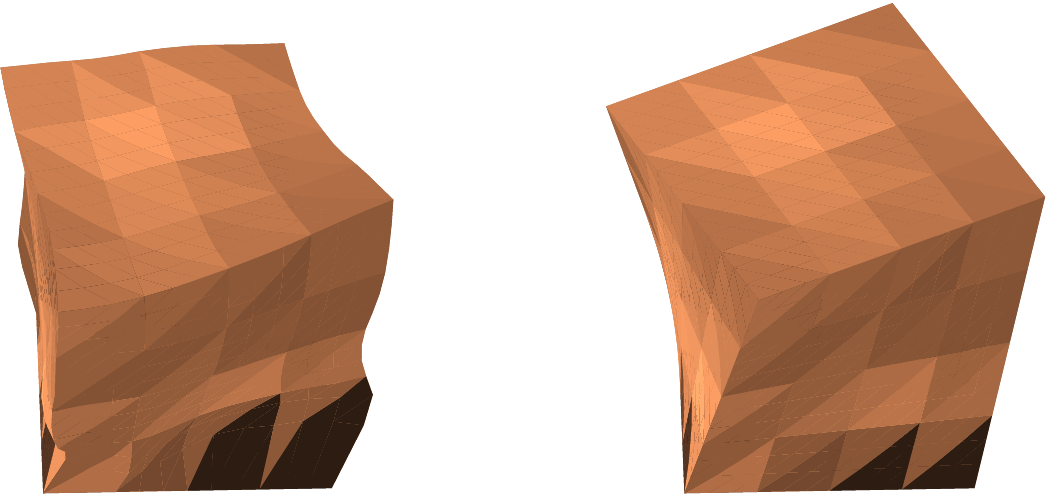}\\
\includegraphics[width=0.6\textwidth]{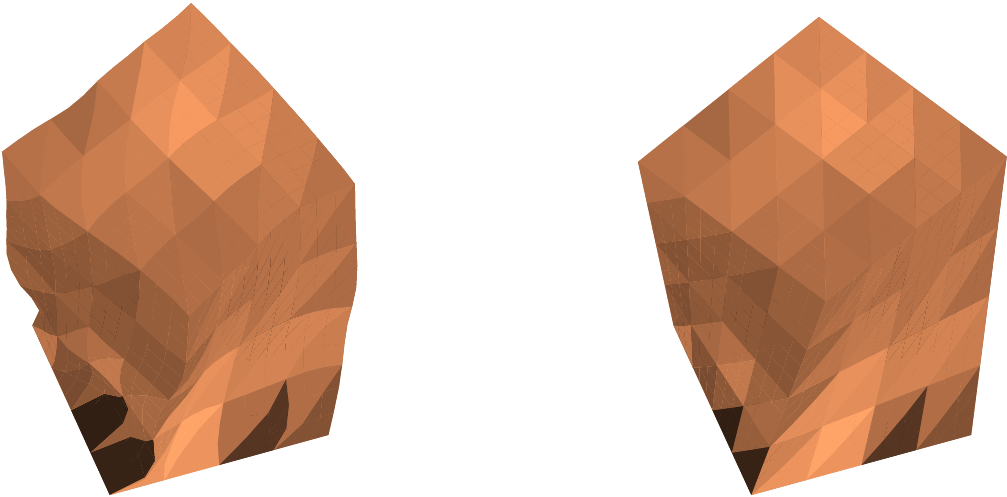}\\
\includegraphics[width=0.6\textwidth]{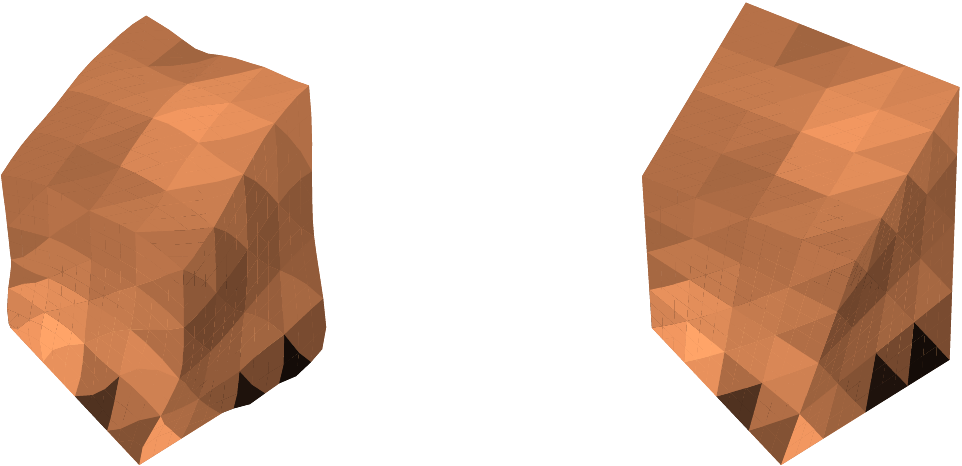}
}
\caption{{\footnotesize Snapshots from the simulation described in the text with the computed solution $\uu_h$ using the optimal control $z_h$ on the left and the desired state $\uu_d$ on the right at timesteps 81, 241, 321, and 401}} \label{fig:5}
\end{figure} 

 {
\section{Conclusion}
In this work, we have studied a PDE constrained optimization problem (or an optimal control problem) 
where the PDE constraints (state 
equations) describe elastic wave propagation in piezoelectric solids. The electric flux acts as the control 
variable and the bound constraints on the control are considered. We enforce the requisite regularity on the
control variable to show well-posedness of the state equations via a cost functional. In addition, we establish
the well-posedness of the optimization problems and derive the first order necessary and sufficient optimality
conditions.
In addition to showing the existence and uniqueness of a semidiscrete (discrete in space continuous in time) optimal control, we have shown the convergence of the semidiscrete optimal control to its continuous counterpart. 
We also provide details on the fully discrete scheme and have given numerical examples in 3-D. 

While the control problem under consideration can be useful in the design of new materials that could be manipulated in response to electric stimuli, it would also be interesting to study other related control problems. For example, in many applications of piezoelectric materials, the weight of people stepping on the
material is used to generate electric current. In this context we would need to use a Dirichlet condition on the elastic displacement as our control variable. This results in a problem that is interesting not only because it incorporates other kinds kinds of applications, but also because the mathematics involved in the problem changes significantly since we need
to incorporate an $H^{1/2}$ norm (in space), for the control, into the cost functional. It will also be interesting to explore other types of control problems in the context of elastic solids with different properties (for example thermoelastic or viscoelastic solids).
}


\bibliographystyle{IMANUM-BIB} 
\bibliography{Refer}

\clearpage

\appendix

%

\section{Proofs of Theorems \ref{thm:1} and \ref{thm:2}} \label{sup:sec:Pfs}

In this Appendix, we present the detailed proofs of Theorems \ref{thm:1} and \ref{thm:2}. 
For an $X$-valued function of a real variable, we consider its antiderivative
\[
(\partial^{-1} f) (t) := \int_0^t f(\tau) \; \mathrm d \tau.
\]

\subsection{The state equation (Theorem \ref{thm:1})}

As a first step, we want to rewrite the state equation in first order form.  To deal with the elliptic equation  { we first introduce the space of gradients of functions in $H_G^1(\Omega)$
\[
\mathcal G(\Omega) = \nabla H_G^1(\Omega) = \{ \nabla \phi : \phi \in H_G^1(\Omega)\}.
\]
We note that since we are imposing the grounding condition, we have that $\nabla: H_G^1(\Omega) \to \mathcal G(\Omega)$ is invertible.  This inverse will be useful in what follows and we will denote it by $g^{-1}$, that is, 
\[
\phi = g^{-1} \mathbf q \in H_G^1(\Omega)  \qquad \Leftrightarrow   \qquad \mathbf q = \nabla \phi \in \mathcal G(\Omega).
\]
Next}, we define the operators  $M_\Omega: \mathrm L^2(\Omega; \matR_\symm) \longrightarrow  {\mathcal G(\Omega)}$ and $M_\Gamma : L_0^2(\Gamma) \longrightarrow  {\mathcal G(\Omega)}$, where $\psi:= {g^{-1}(M_\Omega \mathrm Q+M_\Gamma z)}$ solves 
\begin{alignat*}{4}
&\psi \in H^1(\Omega), \qquad G\psi = 0,\\
&(\KK \nabla \psi, \nabla \varphi)_\Omega = -\langle z, \gamma \varphi \rangle_\Gamma+( {\mathrm Q}, \EE \nabla \varphi)_\Omega \quad \forall \varphi \in H_G^1(\Omega).
\end{alignat*}
We can thus get rid of the electric field (and of the attached elliptic equation) by writing $\psi(t)= {g^{-1}(M_\Omega \boldsymbol\varepsilon(\uu(t))+M_\Gamma z(t))}$, at the same time that we introduce two auxiliary unknowns
\[
\mathrm S:= \partial^{-1} \CC \eps{\uu} \qquad \qquad \mathbf r: = \partial^{-1} \nabla \psi=\partial^{-1}(M_\Omega\boldsymbol\varepsilon(\uu)+M_\Gamma z).
\]
With these definitions, we are ready to formally write the first order formulation of the state equation:
\begin{subequations} \label{sup:eq:FOFst}
\begin{alignat}{4} 
\dot{\uu}(t) &= \rho^{-1}\ddiv (\mathrm S(t) + \EE \mathbf r(t)) &\qquad & t \in [0,T],\label{sup:eq:FOFstA}
\\
\dot{\mathrm S}(t) &= \CC \eps{\uu(t)} && t \in [0,T],
\label{sup:eq:FOFstB}
\\
\dot{\mathbf r}(t) &=  M_\Omega \eps{\uu(t)} + M_\Gamma z(t) && t \in [0,T],\label{sup:eq:FOFstC}\\
\gamma_D \uu(t) &= \mathbf 0 && t \in [0,T],
\label{sup:eq:FOFstD}
\\
\gamma_N (\mathrm S(t) + \EE \mathbf r(t)) & = \mathbf 0 && t \in [0,T],
\label{sup:eq:FOFstE}
\\
\uu(0) &= \mathbf 0, \quad \mathrm S(0) = 0, \quad \mathbf r(0) = \mathbf 0 \label{sup:eq:FOFstF},
\end{alignat}
\end{subequations}
 {where the equations are to be understood in the sense of distributions.}
Our goal now is to analyze \eqref{sup:eq:FOFst} and show it is equivalent to \eqref{eq:state}.  To this end we  define the space $\mathbb H: = \LL_\rho^2(\Omega) \times \mathrm L^2(\Omega; \matR_\symm) \times  {\mathcal G}(\Omega)$ with norm 
\[
\|(\uu, \mathrm S, \mathbf r)\|_\mathbb{H}^2 := (\rho \uu, \uu)_\Omega + (\CC^{-1} \mathrm S, \mathrm S)_\Omega + (\KK \mathbf r, \mathbf r)_\Omega,
\]
where we are using the compliance tensor $\mathcal C^{-1} \in L^\infty(\Omega; \R^{(d \times d) \times (d \times d)})$, whose action is defined as $\matR\ni \mathrm A \mapsto \mathcal C^{-1} \mathrm A := \mathrm B\in L^\infty(\Omega;\matR_{\mathrm{sym}})$ if $\mathcal C \mathrm B : \mathrm M = \mathrm A : \mathrm M$ for every $\mathrm M \in \matR_{\mathrm{sym}}$.  Additionally, we define the space
\[
D(A):= \HH_D^1(\Omega) \times \{ (\mathrm S, \mathbf r)  {\in\mathrm L^2(\Omega; \matR_\symm) \times \mathcal G(\Omega)}  : \ddiv (\mathrm S + \EE \mathbf r) \in \LL^2(\Omega), \gamma_N(\mathrm S + \EE \mathbf r) = \mathbf 0\},
\] 
and the operator $A: D(A) \longrightarrow \mathbb H$ given by
\[
A(\uu, \mathrm S, \mathbf r) := (\rho^{-1} \ddiv(\mathrm S + \EE \mathbf r), \CC \eps{\uu}, M_\Omega \eps{\uu}).
\]
Using the notation $U:= (\uu, \mathrm S, \mathbf r)$ and defining the right-hand side $ F:= (\mathbf 0, 0, M_\Gamma z)$, we can rewrite \eqref{sup:eq:FOFst} as 
\begin{subequations}\label{sup:eq:FOF}
\begin{alignat}{3}
U \in & \CC^0([0,T];D(A))\cap \CC^1([0,T];\mathbb H),\label{eq:A.2.a}\\
\dot{U}(t) & = AU(t) + F(t) \qquad t \in [0,T], \label{eq:A.2.b}\\
U(0) &= 0,
\end{alignat}
\end{subequations}
followed by the postprocessing 
\begin{equation}\label{sup:eq:post}
\psi:=  {g^{-1}(M_\Gamma\boldsymbol\varepsilon(\uu)+M_\Gamma z)}. 
\end{equation}
 {To show that \eqref{sup:eq:FOF} is well-posed, we rely on semigroup theory which requires hypotheses on the operator $A$ and the regularity of the data $F$.} It can be shown that $(AU,U)_\mathbb{H} = 0$ for every $U \in D(A)$ and that the operators $I \pm A:  D(A) \to \mathbb H$ are surjective. We omit the details of these two computations, but note that more information can be found in \cite{BrSaSa2017} by disregarding the acoustic fields. By classical theory of $C_0$-semigroups of operators, this implies that $A$ is the infinitesimal generator of a $\CC_0$-group of isometries in $\mathbb H$  {(see for example \cite[Chapter 1, Theorem 4.3]{Pazy1983} or \cite[Chapter 4, Theorems 4.3 and 5.1]{Showalter1977})}.

 {For a Banach space $X$, we introduce the space 
\[
W^1(X):= \{ f \in \mathcal C^0([0,T]; X) : \dot{f} \in L^1(0,T;X), \dot{f}(0) = f(0) = 0\}.
\] 
We require that $z\in W^1(L_0^2(\Gamma))$, so that $F\in W^1(\mathbb H)$.  The continuity of $F$ implies that $F$ is integrable, and so we have, in the language of \cite{Pazy1983}, a unique mild solution to \eqref{sup:eq:FOF} with the reduced regularity $U \in \mathcal C^0([0,T];\mathbb H)$.  Once we have that $U$ is continuous, we also have that $AU$ is continuous since the spatial operators do not affect the time regularity.  Now using that $F$ is continuous we have that $\dot{U}$ is continuous by \eqref{eq:A.2.b}, and therefore by \cite[Chapter 4, Theorem 2.4]{Pazy1983}, $U$ uniquely solves \eqref{sup:eq:FOF} with the full regularity stated in \eqref{eq:A.2.a}. We also have that $\dot U(0)=0$.

Furthermore, since $A$ is the the infinitesimal generator of a $\CC_0$-group of isometries in $\mathbb H$, we have the bound 
\[
\|U(t)\|_{\mathbb H} \lesssim \int_0^t \|F(\tau)\|_{\mathbb H} \; \mathrm d \tau \lesssim \int_0^t \|z(\tau)\|_\Gamma \; \mathrm d \tau \qquad t \in [0,T].
\]
We obtain a similar bound for $\dot{U}$ because we are requiring $\dot{F} \in L^1(0,T;H)$, and so $\dot{U} \in \mathcal C^0([0,T]; \mathbb H)$ is a mild solution to 
\begin{alignat*}{3}
\ddot{U}(t) &= A\dot{U}(t) + \dot{F}(t) \qquad t \in [0,T],\\
\dot{U}(0) &= 0.
\end{alignat*}
}
The above proves that
\[
\uu \in \mathcal C^1([0,T];\mathbf L_\rho^2(\Omega))\cap
\mathcal C^0([0,T];\mathbf H^1_D(\Omega)),
\qquad \uu(0)=0, \quad \dot\uu(0)=0,
\]
and by \eqref{sup:eq:post}, $\psi\in \mathcal C^0([0,T];H^1_G(\Omega))$,  {with the bounds}
\begin{alignat*}{4}
\|\uu(t)\|_{1,\Omega} &\lesssim \|U(t)\|_\mathbb{H} + \|\dot{U}(t)\|_\mathbb{H} \lesssim \int_0^t \|z(\tau)\|_\Gamma \; \mathrm d \tau + \int_0^t \|\dot{z}(\tau)\|_\Gamma \; \mathrm d \tau,\\
\|\psi(t)\|_{1,\Omega} &\lesssim \|\varepsilon(\uu)(t) \|_{\Omega} + \|z(t)\|_{\Gamma} \lesssim  \int_0^t \|\dot{z}(\tau)\|_\Gamma \; \mathrm d \tau.
\end{alignat*}
Note that we have used Korn's inequality to estimate $\|\bff u(t)\|_{1,\Omega}$ in terms of $\|\bff u(t)\|_\Omega+\|\boldsymbol\varepsilon(\uu)(t)\|_\Omega$.
We also have
\begin{equation}\label{sup:eq:A100}
\langle \dot\uu(t),\,\cdot\,\rangle_\rho
=-(\mathrm S(t)+\mathcal E\bff r(t),\boldsymbol\varepsilon(\,\cdot\,))_\Omega
\qquad \mbox{in $\HH^{-1}_D(\Omega)$,}
\end{equation}
which follows from using \eqref{sup:eq:FOFstA} and \eqref{sup:eq:FOFstE}. Since $\mathrm S+\mathcal E\bff r\in \mathcal C^1([0,T];L^2(\Omega;\matR_\symm))$, then \eqref{sup:eq:A100} implies that $\dot\uu\in \mathcal C^1([0,T];\HH^{-1}_D(\Omega))$ and
\begin{alignat*}{3}
\langle \ddot{\uu}(t), \ww \rangle_\rho &= -(\dot{\mathrm S}(t) + \EE \dot{\mathbf r}(t), \eps{\ww})_\Omega \\
&= - (\CC \eps{\uu(t)} + \EE \nabla \psi(t), \eps{\ww})_\Omega 
\qquad\forall \ww\in \HH^1_D(\Omega).
\end{alignat*}
This and \eqref{sup:eq:post}  {show that \eqref{sup:eq:FOFst} implies \eqref{eq:state}.  The reverse implication follows from integrating the second order form of the equation and defining the auxiliary operators and unknowns defined at the beginning of this section}.  {We finish the proof by remarking that in the statement of the theorem we have taken $z \in \mathcal Z =\{ z \in  H^1(0,T;L_0^2(\Gamma)): z(0) = 0\}$, but we only need to take $z$ in the weaker space $W^1(L_0^2(\Gamma))$.   The result still holds since $\mathcal Z$ is continuously embedded into $W^1(L_0^2(\Gamma))$, and we take $z$ in this stronger space as it allows us to take advantage of its additional structure}.

\subsection{The adjoint equation (Theorem \ref{thm:2})}
 
 {Now consider} \eqref{sup:eq:FOF} with $F(t):=(\bff f(T-t),0,\mathbf 0)$, and note that
\begin{equation}\label{sup:eq:A200}
\| F(t)\|_{\mathbb H}\lesssim \|\bff f(T-t)\|_\Omega,
\qquad
\| AF(t)\|_{\mathbb H}\lesssim \| \boldsymbol\varepsilon(\bff f)(T-t)\|_\Omega.
\end{equation}
This means that $\bff f\in \mathcal C([0,T];\HH^1_D(\Omega))$ implies $F\in \mathcal C([0,T];D(A))$ and \eqref{sup:eq:FOF} has a unique solution  {by \cite[Chapter 4, Corollary 2.6]{Pazy1983}}. We thus just need to prove that
\[
\bff p:=(\partial^{-1}\uu)(T-\cdot), \qquad 
\xi:= {g^{-1}(M_\Omega \boldsymbol\varepsilon(\bff p))}
\]
is the solution to \eqref{eq:adj}. This can be done quickly by first noticing that the differential equations gathered in \eqref{sup:eq:FOF} and the definitions of $\bff p$ and $\xi$, imply that for all $t\in [0,T]$,
\begin{alignat*}{6}
\ddot\pp(t) &=\rho^{-1} \mathrm{div}\,
	(\mathrm S(T-t)+\mathcal E\mathbf r(T-t))+\bff f(t),\\
\mathrm S(T-t)&=\mathcal C\boldsymbol\varepsilon(\pp(t)),\\
\bff r(T-t) &=\nabla \xi(t).
\end{alignat*}	
These equalities can be used to verify the second order differential equation \eqref{eq:adjC}, the Dirichlet condition \eqref{eq:adjE}, and the Neumann condition for the elastic stress \eqref{eq:adjF}. Moreover, $\xi= {g^{-1}(M_\Omega\boldsymbol\varepsilon(\bff p))}$ compiles the elliptic differential equation \eqref{eq:adjD}, the grounding condition \eqref{eq:adjGG}, and the Neumann condition for the electric displacement \eqref{eq:adjH}. 
 
With respect to the bounds, we first use \eqref{sup:eq:A200} to obtain estimates
\[
\| \uu(t)\|_\Omega+\|\mathrm S(t)\|_\Omega
\lesssim \int_{T-t}^T \|\bff f(\tau)\|_\Omega \mathrm d\tau.
\]
Since $\xi=M_\Omega\boldsymbol\varepsilon(\pp)$ and $\mathcal C\boldsymbol\varepsilon(\pp)=\mathrm S(T-\cdot)$, this provides a bound
\[
\| \xi(t)\|_{1,\Omega}
\lesssim \| \boldsymbol\varepsilon(\pp)(t)\|_\Omega
\lesssim \int_t^T \|\bff f(\tau)\|_\Omega \mathrm d\tau.
\]
Finally
\[
\| \pp(t)\|_\Omega
	 \le \int_t^T \| \uu(T-\tau)\|_\Omega \mathrm d\tau
	 \le (T-t) \int_t^T \|\bff f(\tau)\|_\Omega \mathrm d\tau, 
\]
and the proof of Theorem \ref{thm:2} is finished.

\end{document}